\documentclass[12pt,leqno,twoside]{amsart}

\usepackage{latexsym,esint}
\usepackage{url,color}

\theoremstyle{plain}

\def\endproof{\hspace*{\fill}\mbox{\ \rule{.1in}{.1in}}\medskip }

\numberwithin{equation}{section}
\numberwithin{figure}{section}

\newtheorem{theorem}{Theorem}[section]
\newtheorem{corollary}[theorem]{Corollary}
\newtheorem{lemma}[theorem]{Lemma}
\newtheorem{proposition}[theorem]{Proposition}

\newcommand{\R}{\mathbb R} 

\newcommand{\ds}{\displaystyle} 

\theoremstyle{definition}

\newtheorem{remark}[theorem]{Remark}

\numberwithin{equation}{section}
\numberwithin{figure}{section}

\def\XXint#1#2#3{{\setbox0=\hbox{$#1{#2#3}{\int}$ }
\vcenter{\hbox{$#2#3$ }}\kern-.6\wd0}}

\begin{document}

\title[Elastic shells with incompatible strains]
{Models for elastic shells with incompatible strains} 
\author{Marta Lewicka, L. Mahadevan and Mohammad Reza Pakzad}
\address{Marta Lewicka, University of Pittsburgh, Department of Mathematics, 
301 Thackeray Hall, Pittsburgh, PA 15260}
\address{L. Mahadevan, Harvard University, School of Engineering and
  Applied Sciences, and Department of Physics, 
Cambridge, MA 02138}
\address{Mohammad Reza Pakzad, University of Pittsburgh, Department of Mathematics, 
301 Thackeray Hall, Pittsburgh, PA 15260}
\email{lewicka@pitt.edu, lm@seas.harvard.edu, pakzad@pitt.edu}
\subjclass{74K20, 74B20}
\keywords{non-Euclidean plates, nonlinear elasticity, Gamma convergence, calculus of variations}
 
\begin{abstract}
The three-dimensional shapes of thin lamina such as leaves, flowers, feathers, wings etc, are driven by the differential strain 
induced by the relative growth. The growth takes place through
variations in the Riemannian metric,   
given on the thin sheet as a function of location in the central plane and also across its thickness.  The shape 
is then a consequence of elastic energy minimization on the frustrated
geometrical object. Here we provide a rigorous derivation of the  
asymptotic theories for shapes of residually strained thin lamina with nontrivial curvatures, i.e. growing 
elastic shells in both  the weakly and strongly curved regimes, generalizing earlier results for the growth 
of nominally flat plates. The different theories are distinguished by the scaling of the mid-surface curvature 
relative to the inverse thickness and growth strain, and also allow us to generalize the classical F\"oppl-von 
K\'arm\'an energy to theories of prestrained 
shallow shells.  
\end{abstract} 

\maketitle

\section{Introduction}
The physical basis for {morphogenesis} is now classical and elegantly presented in D'arcy Thompson's opus \lq\lq On 
growth and form" as follows {\em "An organism is so complex a thing, and growth so complex a phenomenon, that for 
growth to be so uniform and constant in all the parts as to keep the whole shape unchanged would indeed be an unlikely 
and an unusual circumstance. Rates vary, proportions change, and the whole configuration alters accordingly."} From a 
mathematical and mechanical perspective, this reduces to a simple principle: differential growth in a body leads to 
residual strains that will generically result in  changes in the shape of a tissue, organ or body. Eventually, the growth 
patterns are expected to themselves be regulated by these strains, so that this principle might well be the basis for the 
physical self-organization of biological tissues.  Recent interest in characterizing the morphogenesis of low-dimensional 
structures such as filaments, laminae and their assemblies, is driven by the twin motivations of understanding the origin 
of shape in biological systems and the promise of mimicking them in artificial  mimics \cite{klein, Koehl2008, Sharon2}. 
The results  lie at the interface of biology,  physics and engineering, but they also have a deeply geometric character. 
Indeed, the basic question may be characterized in terms of a variation on a classical theme in differential geometry - 
that of embedding a shape with a given metric  in a space of possibly different dimension \cite{Nash1, Nash2}. However, 
the goal now is not only to state the conditions when it might be done (or not), but also to constructively determine 
the resulting shapes in terms of an appropriate mechanical theory. 

While these issues arise in three-dimensional tissues, the combination of the separation of scales that arises naturally 
in slender structures and the constraints associated with the prescription of growth  laws that are functions of space 
(and time) leads to the expectation that the resulting  theories ought to be variants of classical elastic plate and 
shell theories such as the F\"oppl-von K\'arm\'an or the Donnell-Mushtari-Vlasov theories \cite{Calladine}. That this is the case, 
has been shown for bodies that are initially flat and thin i.e. elastic plates with no initial curvature, using analogies 
to thermoelasticity \cite{Mansfield,Maha}, perturbation  analysis \cite{BenAmar, Sharon2}, and rigorous asymptotic analysis 
\cite{lemapa} that follows a program similar to the derivation of the equations for the nonlinear elasticity of thin plates 
and shells \cite{FJMgeo, FJMM, lemopa7, lemopa_convex, lepa} and a
linearized theory \cite{lepa_noneuclid} for residually strained   
Kirchhoff plates \cite{kirchhoff}. However, most laminae are naturally curved in their strain-free configurations, 
as a consequence of slow relaxation, perhaps following a previous growth history. Since even infinitesimal deformations 
of a curved shell will potentially violate isometry relative to its rest state, one expects that differential growth of 
such an object will likely lead to a variety of possible low dimensional theories depending 
on the relative size of the metric changes imposed on the system. This potential multiplicity of asymptotic theories is 
of course presaged by a similar state of affairs for the derivation of a nonlinear theory of elastic shells \cite{FJMhier, lepa}.  
 
\medskip

Here, we build on the discussion in \cite{Maha, Maha2, lemapa} and present a rigorous derivation of a set of 
asymptotic theories for the shape of {\em residually strained} thin lamina with nontrivial curvatures, i.e. growing elastic 
shells. As our starting point for a similar theory for growing curved shells, we use the observation that it is possible to 
change the shape of a lamina such as a blooming lily petal by driving it via excess growth of the margins relative to the 
interior, rather than via midrib deformations \cite{doorn}.  Previously, a thermoelastic analogy 
\cite{Mansfield} suggested a natural generalization of the Donnell-Mushtari-Vlasov shell theory \cite{Calladine} to 
growing shells \cite{Maha2}, proposed as a mathematical model for
blooming, activated from the initial (transverse) out-of-plane  
displacement $v_0$ of a petal's mid-surface.  When $v_0=0$ the equations  
(\ref{new_Karman}) reduce to the prestrained von K\'arm\'an equations 
(\ref{old_Karman}) proposed in \cite{Maha}. These were rigorously derived in \cite{lemapa} from {\em non-Euclidean elasticity}, 
where the imposed 3d prestrain is given via a Riemannian metric, whose components display the appropriate linear target 
stretching tensor $\epsilon_g$ (of order 2 in shell's thickness $h$),
and the bending tensor $\kappa_g$ (of order 1 in $h$, see
(\ref{ah})).  This leads us to focus on a particular regime of scaling for the prestrain tensor \eqref{qh-gamma} which corresponds, a
posteriori, in all different regimes of shallowness studied here,  to von-K\`arm\'an type theories.

\medskip

It is pertinent to start with a few comments regarding this particular 
choice of the scaling regime. From a mathematical point of view, 
the von-K\`arm\`an regime, where the nonlinear elastic energy per unit thickness 
scales like  $h^4$, usually corresponds to sub-linear theories, i.e. 
the first nonlinear theories which arise  when the
magnitude of forces or of prestrain allows the elastic lamina to cross the threshold of linear behavior and to manifest 
phenomena such as buckling. Since theses sublinear theories are also the least complicated
among the nonlinear theories of plates and shells arising in the  
literature, and are relevant for many applications, they are popular with engineers, physicists and
applied mathematicians. Therefore, in the analysis of nonlinear shallow shell models with growth, it is reasonable to start with 
the von K\'arm\'an regime as here. In contrast, there are a number of technical challenges  that
must be addressed when deriving lower order nonlinear theories using 
$\Gamma$-convergence. Our current study is potentially thus the first of a series that considers the various possible shell theories that result for
various limiting cases of the growth strain, the boundary loading etc.  Indeed, in a forthcoming paper \cite{LMP-new}, we address a 
shallow shell model that arises in a  forcing regime equivalent to the
energy scaling $h^\beta$ for $\beta<4$, where,  analogous to \cite{FJMhier}, 
technical obstacles regarding properties of the Sobolev solutions to
the Monge-Amp\`ere equations must be addressed, 
before establishing the corresponding $\Gamma$-limit result.

\medskip 

 {In Section 2, we  formulate our main results, in terms of a scaling analysis  that leads to the
hierarchy of limiting models as a function of the various prestrain and shallowness regimes.
In Section 3 we argue that for non-flat mid-surface $S$ (of arbitrary
curvature or for the referential out-of-plane displacement $v_0\neq 0$), 
the variationally  correct 2d theory coincides with  the extension of the classical {\em von K\'arm\'an energy} 
to shells, derived in  \cite{lemopa7}.  In the special case  $v_0=0$, this energy still reduces to the functional whose
Euler-Lagrange equations are those derived growing elastic plates in
\cite{Maha}.   In Section 4, we discuss a new shallow shell model,
valid when the radius of curvature of 
the mid-surface is relatively large compared to the thickness. 
This limit leads to a prestrained plate model which inherits the geometric structure of the shallow shell. 
In Section 5, we consider the case where the radius of curvature and
the thickness are comparable in magnitude, and 
appropriately compatible with the order 
of the prestrain tensor. We show that equations 
for a growing elastic shell can be formally derived by pulling back the in-plane and out-of-plane growth tensors $\epsilon_g$ and $\kappa_g$, 
from {\em shallow shells} $(S_h)^h$ with reference mid-surface $S_h$
given by the scaled out-of-plane displacement 
$hv_0$, onto a flat reference configuration. Furthermore, we argue
that this theory for growing elastic shells is also the Euler-Lagrange
equation of the variational limit for 3d nonlinear elastic energies on
$(S_h)^h$.  In Section 6 we discuss the model where the effects of shallowness are dominated by the growth-induced prestrain.  
In this case the limiting energy is impervious to the influence of the
shell geometry, but the effects of growth may not be neglected. This leads to the the generalized von K\'arm\'an equations for a growing flat plate. 
In Section 7, we justify that under our prestrain or growth scaling assumptions, the derived models are the relevant ones when the boundaries are free and no external forces are present.   
Finally, in Section 8, we conclude with a discussion of the present results and prospects for the future.   Since the proofs of the theorems consist of tedious yet minor 
(though necessary) modifications of the arguments detailed in 
\cite{lemopa7, lemapa, lemopa_convex}, we refer the interested reader
to the Appendix, where they are given for completeness.

\section{Preliminaries and scaling limits} 

Let  $ v_0\in \mathcal{C}^{1,1}(\bar\Omega)$  be an out-of-plate displacement on an open, bounded subset 
$\Omega\subset\mathbb{R}^2$, associated with a
family of surfaces, parametrized by $\gamma\in  [0,1]$:
\begin{equation}\label{shallow-alpha-gamma}  
S_\gamma=\phi_\gamma(\Omega), \mbox{ where } \phi_\gamma(x) = \big(x, \gamma v_0(x)\big)
\qquad \forall x=(x_1, x_2)\in\Omega,
\end{equation} 
The unit normal vector to $S_\gamma$ at $\phi_\gamma(x)$ is given by:
$$\vec n^\gamma(x) = 
\frac{\partial_1\phi_\gamma(x) \times \partial_2\phi_\gamma(x)}
{|\partial_1\phi_\gamma(x) \times \partial_2\phi_\gamma(x)|}
= \frac{1}{\sqrt{1+\gamma^2|\nabla v_0|^2}} \big(-\gamma\partial_1v_0(x),
-\gamma \partial_2v_0(x), 1\big) \qquad\forall x\in\Omega. $$
For small $h>0$, we now consider thin plates $\Omega^h = \Omega\times (-h/2, h/2)$
and 3d shells $(S_\gamma)^h$:
\begin{equation}\label{shallow-gamma} 
(S_\gamma)^h = \big\{\tilde \phi_{\gamma}(x, x_3); ~ x\in\Omega, ~ x_3\in(-h/2,
h/2)\big\},
\end{equation}
where the extension $\tilde \phi_\gamma:\Omega^h\rightarrow \mathbb{R}^3$ 
of $\phi_\gamma$ on $\Omega^h$ in (\ref{shallow-alpha-gamma}) is given by the formula:
\begin{equation}\label{kl-gamma}
\tilde\phi_{\gamma}(x, x_3) = \phi_\gamma(x) + x_3\vec n^\gamma(x)
\qquad\forall (x, x_3)\in\Omega^h.
\end{equation}  

For an elastic body with the reference configuration $(S_\gamma)^h$ we assume that its elastic energy density 
$W:\mathbb{R}^{3\times 3}\longrightarrow \mathbb{R}_{+}$ is $\mathcal{C}^2$ regular 
in a neighborhood of $SO(3)$. Moreover, we assume that $W$ satisfies the normalization, 
frame indifference and nondegeneracy conditions:
\begin{equation}\label{en_as}
\begin{split}
\exists c>0\quad \forall F\in \mathbb{R}^{3\times 3} \quad
\forall R\in SO(3) \quad
&W(R) = 0, \quad W(RF) = W(F),\\
&W(F)\geq c~ \mathrm{dist}^2(F, SO(3)).
\end{split}
\end{equation}
where $F=\nabla u$ is the deformation gradient  relative to the reference configuration $(S_\gamma)^h$. 
For prestrained  structures characterized by the  Riemannian metric:
$$p^h = (q^h)^Tq^h \quad \mbox{ on } ~~(S_\gamma)^h,$$  the tensor $F=\nabla u$ is replaced by $F =\nabla u (q^h)^{-1}$, 
so that the thickness averaged elastic energy is given by:
\begin{equation}\label{IhW-gamma}
I^{\gamma,h}(u)  = \frac{1}{h}\int_{(S_\gamma)^h} W(F) ~\mbox{d}z 
= \frac{1}{h}\int_{(S_\gamma)^h} W(\nabla u (q^h)^{-1}) ~\mbox{d}z,
\qquad \forall u \in W^{1,2}((S_\gamma)^h,\mathbb{R}^3).
\end{equation} 
Letting  $\epsilon_g,\kappa_g:\bar\Omega \rightarrow \mathbb{R}^{3\times  3}$ be two given smooth tensors, 
for each small $h$ we define the growth tensors $q^h$ on $(S_\gamma)^h$ by:
\begin{equation}\label{qh-gamma}
q^h(\phi_\gamma(x) + x_3\vec n^\gamma(x)) = \mbox{Id} + h^2\epsilon_g(x) +
hx_3\kappa_g(x) \qquad \forall (x, x_3)\in\Omega^h.
\end{equation}
The corresponding metric $p^h=(q^h)^Tq^h$ on $(S_\gamma)^h$ is then: 
$$p^h(\phi_\gamma(x) + x_3\vec n^\gamma(x)) = \mbox{Id} + 2h^2\mbox{sym }\epsilon_g(x) +
2hx_3\mbox{sym }\kappa_g(x) +\mathcal{O}(h^3).$$

An important part of our study focuses on the asymptotic behavior in the limit of vanishing thickness $h\to 0$ of the
 variational models $I^{\gamma, h}$ in (\ref{IhW-gamma}), when $\gamma= \gamma(h)= h^\alpha$ for a given 
exponent $0 \le \alpha < +\infty$. The regime $\alpha>0$ corresponds to the study of a {\it shallow} shell. 
However, we will identify {\em three distinct shallow shell limit models}, 
depending on the asymptotic behavior of the ratio $\gamma/h$, which in our setting depends only on the value of $\alpha$.  This allows us to 
rigorously derive the $\Gamma$-limits: ${\Gamma\mbox{-}}\lim_{h\to 0} \frac1{h^4} I^{h^\alpha, h}$, 
and show  that under suitable incompatibility conditions on the strain tensors $\epsilon_g$ or $\kappa_g$, the infimum of energies 
$I^{h^\alpha,h}$ scales like $h^4$ irrespective of the value of $\alpha$. 
This justifies our choice of the energy scaling and lends credibility to limiting models as physically relevant in the corresponding scaling regimes. 

\medskip 

To get a sense of our results it is useful to summarize our analysis in terms 
of the $\Gamma-$limit of $\frac1{h^4} I^{h^\alpha, h},$ which can be identified as follows:

\begin{equation}\label{energylimits}
\Gamma\mbox{-}\lim_{h\to 0} \frac1{h^4}  I^{h^\alpha, h} = \left \{ \begin{array}{ll} 
{\mathcal I}_4 & \mbox{if} \quad \alpha=0 \\ \\
{\mathcal I}^\infty_4 & \mbox{if} \quad 0<\alpha <1 \\ \\
{\mathcal I}^1_4 & \mbox{if} \quad \alpha=1 \\ \\
{\mathcal I}^0_4 & \mbox{if} \quad \alpha>1.
\end{array} \right. 
\end{equation} 
The above four theories collapse into one and the same theory when $v_0=0$. Otherwise we must 
deal with four distinct potential limits depending on the choice of parameters, in the following order: 

\medskip 

{\bf Case 1. $\alpha=0$.} This corresponds to $\gamma = 1$ where the 3d model 
is that of the prestrained non-linear elastic shell of arbitrarily large curvature (no shallowness involved).  
We will show that the $\Gamma$-limit in this case leads to a prestrained von K\'arm\'an model ${\mathcal I}_4$ 
for the 2d mid-surface $S_1$. This will be described in a more general framework in Section \ref{sec2}. 

\medskip 

{\bf Case 2. $0<\alpha<1$.} This corresponds to the flat limit $\gamma\to 0$  when the energy 
can be conceived as a limit of the von K\'arm\'an models $\mathcal I_4$ for shallow shells $S_\gamma$. 
In other words, this limiting model corresponds to the case when:
$ \displaystyle \lim_{h\to 0} \frac{\gamma(h)}{h} = \infty, $  
and it can be also  identified as:
$$ \ds {\mathcal I}^\infty_4 = \Gamma\mbox{-}\lim_{\gamma\to 0} \Big (\Gamma\mbox{-}\lim_{h\to 0} 
\frac 1{h^4} I^{\gamma,h}\Big ), $$ 
by choosing the distinguished sequence of limits, first as $h \rightarrow 0$ and then  $\gamma \rightarrow 0$.
In Section \ref{sec4} we will see that ${\mathcal I}^\infty_4$ is formulated for displacements of a plate but it inherits 
certain geometric properties of  shallow shells $S_\gamma$, such as the first-order infinitesimal isometry constraint. 

\medskip 

{\bf Case 3. $\alpha=1$.} This corresponds to the case 
$\ds \lim_{h\to 0} \gamma(h) /h = 1$. The limit model $\mathcal{I}_4^1$,  derived in Section \ref{sec3}, 
is an unconstrained energy minimization, reflecting both the effect 
of shallowness and that of the prestrain. It corresponds to a simultaneous passing to the limit $(0,0)$ 
of the pair $(\gamma, h)$ in (\ref{IhW-gamma}). The Euler-Lagrange equations (\ref{new_Karman}) of $\mathcal{I}_4^1$ 
were suggested in \cite{Maha2} for the description of the deployment of petals during the blooming of a flower. 

\medskip 

{\bf Case 4. $\alpha>1$.} Finally, the $\Gamma$-limit for all values of $\alpha>1$, 
i.e. when $\ds \lim_{h\to 0} \frac {\gamma(h)}{h} =0,$ coincides with the zero thickness limit of the degenerate case 
$\gamma=0$, which is  the prestrained plate von K\'arm\'an model, discussed in \cite{lemapa}. 
This limiting energy can be obtained by taking the consecutive limits: 
$$ \ds {\mathcal I}^0_4=  \Gamma\mbox{-}\lim_{h\to 0} \Big (\Gamma\mbox{-}\lim_{\gamma\to 0} 
\frac 1{h^4} I^{\gamma,h}\Big ).$$ 

\section{The prestrained von K\'arm\'an energy for shells of arbitrary  curvature: $\alpha =0$} \label{sec2}

When the parameter $\alpha=0$, the 3d variational problem associated with (\ref{IhW-gamma}) is  reduced 
to the 3d nonlinear elastic energy on the thin shell $S^h_1$, where $S_1$ is the graph of $v_0$.
 It is useful to discuss this model in a more general framework.  Let $S$ be an arbitrary  2d surface
embedded in $\mathbb{R}^3,$ that is compact, connected, oriented, and of class $\mathcal{C}^{1,1}$.
The boundary $\partial S$ of $S$ is assumed to be the union of finitely many (possibly none) 
Lipschitz continuous curves. We consider the family $\{S^h\}_{h>0}$ of thin shells of thickness 
$h$ around $S$: 
$$ S^h = \left\{z=x + t\vec n(x); ~ x\in S, ~ -\frac{h}{2}< t <  \frac{h}{2}\right\}, \qquad 0<h<h_0 \ll 1  $$
where we use the following notation: $\vec n(x)$ for the unit normal, $T_x S$ for the tangent space,
and $\Pi(x) = \nabla \vec n(x)$ for the shape operator on $S$, at a given $x\in S$. 
The projection onto $S$ along $\vec n$ is denoted by $\pi$, so that
$\pi(z) =x$ for all $ z=x+t\vec n(x)\in S^h,$ and we assume that $h \ll 1$ is small enough to have 
$\pi$ well defined on each $S^h$.

The instantaneous growth of $S^h$ is described by smooth tensors: 
$\epsilon_g,\kappa_g:\overline S\longrightarrow \mathbb{R}^{3\times 3}$, by:
\begin{equation}\label{ah} a^h=[a_{ij}^h] :\overline{S^h}\longrightarrow\mathbb{R}^{3\times 3}, \qquad
 a^h(x + t\vec n)=\mathrm{Id} + h^2\epsilon_g(x) + ht\kappa_g(x).
\end{equation}
The growth tensor $a^h$ is as in \cite{Maha, lemapa}, now in a general non-flat geometry setting.
Given the elastic energy density $W:\mathbb{R}^{3\times 3}\longrightarrow \mathbb{R}_{+}$ as in (\ref{en_as}), 
the thickness averaged elastic energy induced by the prestrain $a^h$ is given by:
\begin{equation}\label{IhW}
I^h(u^h)  = \frac{1}{h}\int_{S^h} W(\nabla u^h(a^h)^{-1}) ~\mbox{d}z,
\qquad \forall u^h\in W^{1,2}(S^h,\mathbb{R}^3).
\end{equation}

\medskip

Taking the asymptotic limit (the $\Gamma$-limit as $h\to 0$, see Theorem \ref{thmainuno} 
and Theorem \ref{thmaindue} below) of the energies
$I^h$ (note that $I^h=I^{1,h}$ in the notation of (\ref{IhW-gamma}))
then leads to the variationally correct model for weakly prestrained shells.
It corresponds to the following nonlinear energy functional
$\mathcal{I}_4$ acting on the admissible limiting pairs $(V, B)$: 
\begin{equation}\label{vonK}
\begin{split}
\forall V\in\mathcal{V} \quad \forall B\in\mathcal{B} \qquad
\mathcal{I}_4(V,B)= &\frac{1}{2} 
\int_S \mathcal{Q}_2\left(x,B - \frac{1}{2} (A^2)_{tan} - (\mathrm{sym}~\epsilon_g)_{tan}\right)\\
&+ \frac{1}{24} \int_S \mathcal{Q}_2\Big(x,(\nabla(A\vec n) -
  A\Pi)_{tan} - (\mathrm{sym}~\kappa_g)_{tan}\Big).
\end{split}
\end{equation}

\medskip

Here,  {\underline{the space $\mathcal{V}$ consists of {\em the first-order infinitesimal isometries} on $S$}}, defined by:
\begin{equation}\label{spaceV}
\mathcal{V} = \left\{V\in W^{2,2}(S,\mathbb{R}^3); ~~
\tau\cdot \partial_\tau V(x) = 0 \quad  \forall {\rm{a.e.}} \,\, x\in S \quad
\forall\tau\in T_x S\right\},
\end{equation}
that is those $W^{2,2}$ regular displacements $V$ for whom the change of metric on $S$ due to the deformation $\mbox{id} +
\epsilon V $ is of order $\epsilon^2$, as $\epsilon\to 0$. Furthermore, for a matrix field $A\in L^2(S, \mathbb{R}^{3\times
  3})$,  let $A_{tan}(x)$ denote the tangential minor of $A$ at $x\in S$, that is $[(A(x)\tau)\eta]_{\tau,\eta\in T_x S}$. The skew-symmetric
gradient of $V$ as in (\ref{spaceV}) then uniquely determines a $W^{1,2}$ matrix field $A:S\longrightarrow so(3)$  
so that: $\partial_\tau V(x) = A(x)\tau$ for all $\tau\in T_x S$. Hence, we equivalently write:
\begin{equation*}\label{Adef-intro} 
\begin{split}
\mathcal{V} = \big\{V\in W^{2,2}(S,\mathbb{R}^3); \quad &\exists A\in
  W^{1,2}(S,\mathbb{R}^{3\times 3}) \quad \forall {\rm{a.e.}} \,\, x\in S \quad \forall \tau\in T_x S\\
& \qquad\qquad \partial_\tau V(x) = A(x)\tau \mbox{ and }  A(x)^T= -A(x) \big\}.
\end{split}
\end{equation*}

For a plate, that is when $S\subset\mathbb{R}^2$, an equivalent analytic characterization for $V=(V^1, V^2, V^3) \in {\mathcal V}$ is given by:  
$(V^1, V^2) =  (-\omega y, \omega x) + (b_1, b_2)$, while the out-of-plane displacement $V^3\in W^{2,2}(S,\mathbb{R})$ remains unconstrained.

\medskip

{\underline{The space $\mathcal{B}$ in (\ref{vonK}) consists of {\em finite strains}}:
\begin{equation}\label{fss}
\mathcal{B} = \Big\{L^2 - \lim_{\epsilon\to 0}\mathrm{sym }\nabla w^\epsilon; 
~~ w^\epsilon\in W^{1,2}(S,\mathbb{R}^3)\Big\},
\end{equation}
which are all limits of symmetrized gradients of sequences of displacements on $S$. By $\mathrm{sym } \nabla w (x)$ we mean here a bilinear form on
$T_xS$ given by: $(\mathrm{sym }\nabla w(x)\tau)\eta = \frac{1}{2}[(\partial_\tau w(x))\eta+(\partial_\eta w(x))\tau]$ for all $\tau, \eta \in T_xS$.

It follows (via Korn's inequality) that for a flat plate $S\subset \mathbb{R}^2$, the space $\mathcal{B}$ consists precisely of
symmetrized gradients of all the in-plane displacements: 
$\mathcal{B} =
\{\mathrm{sym }\nabla w; ~~ w\in W^{1,2}(S,\mathbb{R}^2)\}$.
When $S$ is strictly convex, rotationally symmetric, or
developable without flat regions, it has been proven in \cite{lemopa7, schmidt} that
$\mathcal{B}=L^2(S,\mathbb{R}^{2\times 2}_{sym})$, i.e. it contains
all symmetric matrix fields on $S$ with square integrable entries.

\medskip

Finally, in (\ref{vonK}), {\underline{the quadratic forms}:
\begin{equation}\label{quad}
\mathcal{Q}_3(F) = D^2 W(\mbox{Id})(F,F), \quad
\mathcal{Q}_2(x, F_{tan}) = \min\{\mathcal{Q}_3(\tilde F); ~~ \tilde
F\in\mathbb{R}^{3\times 3}, ~~ (\tilde F - F)_{tan} = 0\}.
\end{equation}
where the form $\mathcal{Q}_3$ is defined for  all $F\in\mathbb{R}^{3\times 3}$, while 
$\mathcal{Q}_2(x,\cdot)$ for a given $x\in S$ is defined on tangential minors 
$F_{tan}$ of such matrices. Both forms $\mathcal{Q}_3$ and all $\mathcal{Q}_2(x,\cdot)$ 
are nonnegative definite and depend only on the  symmetric parts of their 
arguments.

\medskip

We now have the following results, stating in particular that the functional $\mathcal{I}_4$ is the 
$\Gamma$-limit \cite{dalmaso} of the scaled energies $h^{-4}I^h$:

\begin{theorem}\label{thmainuno}
Let  a sequence of deformations $u^h\in W^{1,2}(S^h,\mathbb{R}^3$) satisfy $I^h(u^h) \leq Ch^4$. 
Then there exists proper rotations $\bar R^h\in SO(3)$ and 
translations $c^h\in\mathbb{R}^3$ such that for the renormalized deformations: 
$$
y^h(x+t\vec n(x)) = (\bar R^h)^T u^h(x+t\frac{h}{h_0}\vec n) - c^h:S^{h_0}\longrightarrow\mathbb{R}^3
$$ defined on the common thin shell $S^{h_0}$, the following holds.
\begin{itemize}
\item[(i)] $y^h$ converge in $W^{1,2}(S^{h_0},\mathbb{R}^3)$ to $\pi$.
\item[(ii)] The scaled displacements:
\begin{equation}\label{Vh}
V^h(x)={h}^{-1}\fint_{-h_0/2}^{h_0/2}y^h(x+t\vec n) - x~\mathrm{d}t
\end{equation}
converge (up to a subsequence) in $W^{1,2}(S,\mathbb{R}^3)$ to some
$V\in \mathcal{V}$.
\item[(iii)]  The scaled averaged strains:
\begin{equation}\label{Bh}
B^h(x) = {h}^{-1} \mathrm{sym}\nabla V^h(x)
\end{equation}
converge (up to a subsequence) 
weakly in $L^{2}(S,\mathbb{R}^{2\times 2})$ to a limit $B\in\mathcal{B}$.
\item[(iv)] The lower bound holds:
$$\liminf_{h\to 0} {h^{-4}} I^h(u^h) \geq \mathcal{I}_4(V,B).$$
\end{itemize}
\end{theorem}

\begin{theorem}\label{thmaindue}
For every couple $V\in \mathcal{V}$ and  $B\in\mathcal{B}$, there exists a sequence of deformations $u^h\in W^{1,2}(S^{h},\mathbb{R}^3)$ such that:
\begin{itemize}
\item[(i)] The rescaled sequence 
$y^h(x + t\vec n) = u^h(x +  t\frac{h}{h_0}\vec n)$ converges in $W^{1,2}(S^{h_0},\mathbb{R}^3)$ to $\pi$.
\item[(ii)] The displacements $\displaystyle V^h$ as in (\ref{Vh}) converge in $W^{1,2}(S,\mathbb{R}^3)$ to $V$.
\item[(iii)] The strains $B^h$ as in (\ref{Bh}) converge in $W^{1,2}(S,\mathbb{R}^{2\times 2})$ to $B$.
\item[(iv)] There holds:
$$\lim_{h\to 0} h^{-4} I^h(u^h) = \mathcal{I}_4(V,B).$$
\end{itemize}
\end{theorem}

The proofs follow through a combination of arguments in \cite{lemapa} and \cite{lemopa7}, which we do not repeat here but instead comment on the 
functional (\ref{vonK}) and its relation with the prestrained von K\'arm\'an equations for plates. 

\medskip

Here, in analogy with the theory for  flat plates $S\subset \mathbb{R}^2$ with incompatible strains \cite{lemapa}, in (\ref{ah}) we have  
assumed that the target metric is $2$nd order in thickness $h$ for the in-plane  stretching $(\mbox{sym } \epsilon_g)$, and  $1$st order in $h$ for
bending $(\mbox{sym } \kappa_g)$.  Due to this particular choice of scalings the limit energy $\mathcal{I}_4$ is composed of exactly two terms,  corresponding to 
stretching and bending. The argument of the integrand in the first term, namely  $B - \frac{1}{2} (A^2)_{tan} -
(\mathrm{sym}~\epsilon_g)_{tan}$, represents the difference of the  second order stretching 
induced by  the deformation $v^h=\mbox{id} + h V + h^2 w^h$ from the target stretching  $(\mbox{sym}~ \epsilon_g)$, 
with $V\in \mathcal V$ and $\mbox{sym} \nabla w^h \to B$.  The argument of the integrand in the second term
$(\nabla(A\vec n) -  A\Pi)_{tan} -(\mathrm{sym}~\kappa_g)_{tan}$, represents
the difference of the first order bending induced by $v^h$ from the target bending $(\mbox{sym}~ \kappa_g)$. 

In general, the second order displacement $w$ can be very
oscillatory. Due to the non-trivial geometry of the mid-surface $S$, the
finite strain space $\mathcal B$ is usually large and hence a bound on the $L^2$ norm of the symmetric gradients 
$\mbox{sym} \nabla w^h$ implies only a very weak bound on $w^h$. 
The limiting tensor $B$ can hence be written only as the symmetric 
gradient of a very weakly regular distribution (not a classical higher order displacement).

\begin{remark}
When the mid-surface $S$ is elliptic, then for any first order isometry
$V\in \mathcal{V}$ there exists  $B\in \mathcal{B} = 
L^2(S,\mathbb{R}_{sym}^{2\times 2})$ such that $B- \frac{1}{2}
(A^2)_{tan} - (\mbox{sym } \epsilon_g)_{tan} = 0$ (see \cite{lemopa_convex}). This implies that 
for any $V$ there exists  a higher order
modification  $w^h$ for which in the limit, the second order 
target stretching is realized.  Thus,  the energy $\mathcal{I}_4$ reduces to:
$$\mathcal{I}_4(V) =  \frac{1}{24} \int_S \mathcal{Q}_2\Big(x,(\nabla(A\vec n) -
  A\Pi)_{tan} - (\mathrm{sym}~\kappa_g)_{tan}\Big)~\mbox{d}x, $$
i.e. the bending term which is to be minimized over the space $\mathcal{V}$.
Note that this variational problem is convex (minimizing a convex integral over a
linear space $\mathcal{V}$), and hence it admits only one 
solution (up to rigid motions).   Following the analysis in
\cite{lemopa_convex}, we see that for elliptic surfaces, all limiting 
theories for $h^{-\beta}I^h$ under the energy scaling $\beta>2$,  coincide with the linear theory
$\mathcal{I}_4$ as above, while the sublinear theory, to be used in
the description of buckling, is the Kirchhoff-like (nonlinear bending)
theory corresponding to 
$\beta=2$ and derived in \cite{lepa_noneuclid}.
\end{remark}

\section{The prestrained shallow shell with a first-order isometry constraint: $0 <\alpha<1$}\label{sec4}
 
When the parameter $0<\alpha<1,$  the highest order terms (of order $h^{2\alpha}$) in the prestrain metric $p^h$ on $(S_\gamma)^h$ pulled back on the flat reference 
configuration $\Omega^h$, turn out to be \lq \lq compatible'', i.e. entirely generated by the reference 
displacement $h^\alpha v_0$.
In other words, the shallow shell will easily compensate for these terms by rigidly keeping its
structure at the $h^\alpha$ order and only will make adjustments  
at higher orders to the prestrain induced by $\epsilon_g$ and
$\kappa_g$. In the limit as $h\to 0$ we therefore expect that the effective energy functional on $\Omega$ will depend only 
on the out-of-plane and the in-plane displacements of respective orders $h$ and $h^2$. 
Yet, as we shall see below, the residual curvature of mid-surfaces will appear in a two-fold manner: 
as a linearized first-order isometry
constraint on the out-of-plate displacement  (\ref{constr2}),  and also as a defining constraint on the space 
of admissible in-plane displacements. 
The mid-plate $\Omega$ will inherit the space of first order infinitesimal isometries (\ref{spaceV}) 
and the finite strain space (\ref{fss}), in the asymptotic limit of vanishing curvature shells. 

\medskip

{\underline{The space of {\em finite strains} ${\mathcal B}_{v_0}\subset L^2(\Omega,\mathbb{R}^{2\times 2}_{sym})$}  
is defined as: 
$$ \mathcal{B}_{v_0}= \Big\{L^2 - \lim_{\epsilon\to 0}\big(\mathrm{sym }\nabla w^\epsilon 
+ \mathrm{sym}(\nabla v^\epsilon \otimes \nabla
v_0)\big);  ~~ w^\epsilon\in W^{1,2}(\Omega,\mathbb{R}^2), ~v^\epsilon \in W^{1,2}(\Omega,\mathbb{R}) \Big\}.$$
We now identify ${\mathcal B}_{v_0}$ with each of the finite strain spaces of the shallow surfaces $S_\gamma$:

\begin{lemma}\label{fssh} 
Let the surfaces $S_\gamma$ be as in (\ref{shallow-alpha-gamma}). Then for all $\gamma\neq 0$, the finite strain spaces: 
\begin{equation*}
\mathcal{B}^\gamma = \Big\{L^2 - \lim_{\epsilon\to 0}\mathrm{sym }\nabla w^\epsilon; 
~~ w^\epsilon\in W^{1,2}(S_\gamma,\mathbb{R}^3)\Big\},
\end{equation*} 
are each isomorphic to ${\mathcal B}_{v_0}$ via the linear isomorphism:
$$ {\mathcal T^\gamma}: L^2(S_\gamma, {\mathcal L}^2_{sym} (TS_\gamma,\R)) \to
L^2(\Omega, \R^{2\times 2}_{sym}).  $$ 
Here, $L^2(S_\gamma, {\mathcal L}^2_{sym} (TS_\gamma,\R))$ is the space of all $L^2$-sections of the bundle of 
symmetric bilinear forms on $S_\gamma$, and ${\mathcal T^\gamma}$ is naturally defined by: 
$$ [{\mathcal T}^\gamma(\sigma)(x)]_{ij}=  \sigma(\phi_\gamma(x)) (\partial_i \phi_\gamma(x), 
\partial_j \phi_\gamma(x)) \quad \forall \,{\mbox a.e.} 
\, x\in \Omega \quad \forall \sigma \in 
L^2(S_\gamma, {\mathcal L}^2_{sym} (TS_\gamma,\R)).  $$
\end{lemma}
\begin{proof}
Let $w\in W^{1,2}(S_\gamma, \mathbb{R}^3)$ and write $\tilde w =(\tilde
w_1, \tilde w_2, \tilde w_3) = w\circ\phi_\gamma\in
W^{1,2}(\Omega,\mathbb{R}^3)$. Then, for $i,j=1,2$ we have:
$$(\mbox{sym}\nabla w)(\partial_i \phi_\gamma, \partial_j\phi_\gamma) =
\frac{1}{2}\left(\partial_i\tilde w\cdot \partial_j\phi_\gamma
  + \partial_j\tilde w\cdot \partial_i\phi_\gamma\right)
= \Big[\mbox{sym}\nabla (\tilde w_1, \tilde w_2) +
\gamma~\mbox{sym}(\nabla \tilde w_3\otimes \nabla v_0)\Big]_{ij}.$$
Take now a sequence $w^\epsilon\in W^{1,2}(S_\gamma, \mathbb{R}^3)$ such
that $\lim_{\epsilon\to 0}\mbox{sym}\nabla w^\epsilon = B_\gamma\in
\mathcal{B}^\gamma$. Then:
$$\mathcal{T}^\gamma(B_\gamma) = \lim_{\epsilon\to  0}\mathcal{T}^\gamma(\mbox{sym}\nabla w^\epsilon) = 
\lim_{\epsilon\to 0}\Big(\mbox{sym}\nabla (\tilde w_1^\epsilon, \tilde
w_2^\epsilon) + \mbox{sym}(\nabla(\gamma\tilde w_3^\epsilon)\otimes
\nabla v_0)\Big) \in\mathcal{B}_{v_0},$$
which proves the claim.
\end{proof}

The following is a consequence of Lemma \ref{fssh}, \cite[Lemma 5.6]{lemopa7} and \cite[Lemma 3.3]{schmidt}: 

\begin{corollary}\label{cor-ellip}
Assume that:
\begin{itemize}
\item[{(i)}] either: $v_0\in {\mathcal C}^{2,1}(\Omega)\cap
  \mathcal{C}^{1,1}(\bar\Omega)$ and $\det \nabla^2 v_0 \geq c>0$ in $\Omega$,  
\item[{(ii)}] or: $v_0\in\mathcal C^{2}(\bar \Omega)$ with
  $\det\nabla^2 v_0=0$ in $\Omega$, and $\nabla^2v_0$ does not vanish
  identically on any open region in $\Omega$.
\end{itemize}
Then: 
\begin{equation}\label{B=all}
{\mathcal B}_{v_0} = L^2(\Omega, \R^{2\times 2}_{sym}).
\end{equation} 
\end{corollary}
Indeed, in \cite{lemopa_convex} we proved that for any strictly
elliptic surface $S$, its finite strain space $\mathcal{B}$ equals 
$L^2(\Omega, \R^{2\times 2}_{sym})$. Since every $S_\gamma$ is strictly
elliptic under the assumption (i), the result follows by the
equivalence of spaces $\mathcal{B}^\gamma$ and $\mathcal{B}_{v_0}$ in Lemma
\ref{fssh}. The same observation can be derived directly, as follows. Given
$B:\Omega\rightarrow \mathbb{R}_{sym}^{2\times 2}$ smooth enough, we
first solve for $v$ in:
\begin{equation}\label{equation-mystery} 
\left\{\begin{array}{ll}
\mbox{cof}\,\, \nabla^2 v_0 : \nabla^2 v= -\textrm{curl}^T
\textrm{curl} \,B  & \mbox{ in } \Omega,\\
v=0 & \mbox{ on } \partial\Omega.
\end{array} \right.
\end{equation} 
Then we have:  
$$\textrm{curl}^T \textrm{curl}\, B =  -\mbox{cof}\nabla^2 v :
\nabla^2v_0 = 
\mbox{curl}^T\mbox{curl} (\nabla v\otimes\nabla v_0) = 
\textrm{curl}^T \textrm{curl} \Big ( \mbox{sym} (\nabla v \otimes
\nabla v_0)\Big )$$ 
(see also Remark \ref{rem4.7}), and therefore:
$$ B =  \mbox{sym} \nabla (v_1, v_2) + \mbox{sym} (\nabla v \otimes \nabla v_0), $$ 
for some in-plane displacement $(v_1, v_2):\Omega\rightarrow \mathbb{R}^2$.
The density of smooth fields $B$ in the space $L^2(\Omega, \R^{2\times
  2}_{sym})$ now yields the result.  

\begin{remark}
We expect that the property (\ref{B=all}) is satisfied for a generic $v_0$, whenever $\nabla^2 v_0$ 
does not vanish identically on any open region of $\Omega$. The 
argument requires studying very weak solutions of 
the mixed-type equation (\ref{equation-mystery}). When  this equation is 
degenerate ($v_0\equiv 0$), ${\mathcal B}_{v_0}$ coincides with the space of 
all matrix fields in the kernel of the operator $\mathrm{curl}^T
\mathrm{curl}$ and hence it is only a proper subset of $L^2(\Omega,
\R^{2\times 2}_{sym})$, consisting of symmetric gradients. 
\end{remark} 

We now present the main $\Gamma$-convergence result for the shallow
shell regime $0<\alpha<1$. The proofs which consist of tedious
modifications of the arguments in \cite{lemopa7, lemapa}, are outlined
in the Appendix.

\begin{theorem}\label{shallow<1-liminf}
Let $0<\alpha<1$.  
Assume $u^h\in W^{1,2}((S_{h^\alpha})^h,\mathbb{R}^3)$ satisfies
$I^{h^\alpha, h} (u^h)\leq Ch^{4}$, where $I^{\gamma,h}$ is given as in (\ref{IhW-gamma}). 
Then there exists $\bar R^h\in SO(3)$ and $c^h\in\mathbb{R}^3$ such that for the normalized
deformations:
\begin{equation*}\label{resc-gamma}
y^h(x,t) = (\bar R^h)^T (u^h\circ \tilde \phi_{h^\alpha})(x, ht) - c^h : \Omega^1\longrightarrow\mathbb{R}^3
\end{equation*} 
with $\phi_\gamma$ and $\gamma=h^\alpha$ as in (\ref{shallow-alpha-gamma}), we have:
\begin{itemize}
\item[(i)] $y^h(x,t)$ converge in $W^{1,2}(\Omega^1,\mathbb{R}^3)$ to
  $x$.
\item[(ii)] The scaled displacements $V^h(x) = h^{-1}
  \fint_{-1/2}^{1/2}y^h(x,t) - x - h^\alpha v_0 (x)e_3 ~\mathrm{d}t$ converge (up to a
  subsequence) in $W^{1,2}(\Omega,\mathbb{R}^3)$ to 
$(0,0,v)^T$ where $v\in W^{2,2}(\Omega, \mathbb{R})$ and:
\begin{equation}\label{constr2}
\mathrm{cof}\,\,\nabla^2 v_0 : \nabla^2 v=0 \quad \mbox{ in } \Omega.
\end{equation}
\item[(iii)] The scaled strains: 
$$ \displaystyle B_h= \frac 1h \Big (\mathrm{sym} \nabla (V^h_1,
V^h_2) + h^\alpha \mathrm{sym} (\nabla V^h_3 \otimes \nabla v_0) \Big ) $$ 
converge (up to a subsequence) weakly in $L^2$ to some $B \in {\mathcal B}_{v_0}$.
\item[(iv)] Moreover: $\liminf_{h\to 0} h^{-4} I^{h^\alpha, h}(u^h) \geq
  \mathcal{I}_4^\infty (v, B)$, where:
\begin{equation}\label{energy-shallow<1}
\mathcal{I}_4^\infty(v, B)= 
\int_\Omega \mathcal{Q}_2\left(B + \frac{1}{2} \nabla v \otimes \nabla
  v - (\mathrm{sym }~\epsilon_g)_{tan}\right) + \frac{1}{24} 
\int_\Omega \mathcal{Q}_2\Big(\nabla^2 v + (\mathrm{sym }~\kappa_g)_{tan} \Big),
\end{equation}
with $\mathcal{Q}_2$ defined in (\ref{quad}).
\end{itemize}
\end{theorem}

\begin{theorem}\label{shallow<1-limsup}
Let $0<\alpha<1$. For every $v\in W^{2,2}(\Omega,\mathbb{R})$ satisfying (\ref{constr2})
and every $B\in {\mathcal B}_{v_0}$, there exists a sequence of deformations $u^h\in
W^{1,2}((S_{h^\alpha})^h,\mathbb{R}^3)$ such that:
\begin{itemize}
\item[(i)] The sequence $y^h(x,t) = u^h(x+h^\alpha v_0(x)e_3 + ht\vec
  n^{\gamma}(x))$ converges in $W^{1,2}(\Omega^1)$ to $x$.
\item[(ii)] The scaled displacements $V^h$ as in (ii) Theorem \ref{shallow<1-liminf}
converge in $W^{1,2}$ to $(0,0,v)$.
\item[(iii)] The scaled strains $B^h$ as in (iii) Theorem \ref{shallow<1-liminf}
converge weakly in $L^2$ to $B$.
\item[(iv)] $\lim_{h\to 0} h^{-4}I^{h^\alpha, h}(u^h) = \mathcal{I}_4^\infty(v,B).$
\end{itemize}
\end{theorem}  
  
In the special cases of Corollary \ref{cor-ellip}, we have:
\begin{theorem}\label{shallow<1-limsup2}
Assume additionally that $v_0$ is such that (\ref{B=all}) holds.
Then, for every $v\in W^{2,2}(\Omega,\mathbb{R})$ satisfying (\ref{constr2}),
there exists a sequence $u^h\in W^{1,2}((S_{h^\alpha})^h,\mathbb{R}^3)$ such
that (i) and (ii) of Theorem \ref{shallow<1-limsup} hold, and moreover:
$$\ds \lim_{h\to 0} h^{-4}I^{h^\alpha, h}(u^h) 
= \frac{1}{24} \int_\Omega \mathcal{Q}_2\Big(\nabla^2 v + (\mathrm{sym}~\kappa_g)_{tan}\Big ).$$
\end{theorem}  

\begin{remark}\label{rem4.7}
Comparing functionals (\ref{energy-shallow<1}) with (\ref{vonK}), note that the space
$\mathcal{V}(S_\gamma)$ of first-order infinitesimal isometries on $S_\gamma$
is made of displacements $V:S_\gamma\rightarrow \mathbb{R}^3$ of the form:
\begin{equation}\label{VSh}
\begin{split}
&V(\phi_\gamma(x)) = (\gamma v_1(x), h^\alpha v_2(x), v_3) \qquad \forall x\in\Omega,\\
&\mbox{such that }~ (v_1, v_2, v_3)\in W^{2,2}(\Omega,\mathbb{R}^3) 
~\mbox{ and }~ \mbox{sym}\nabla(v_1, v_2) + \mbox{sym}(\nabla v_3\otimes \nabla v_0)=0.
\end{split}
\end{equation}
Indeed, similarly as in the proof of Lemma \ref{fssh},  the condition $\mbox{sym}\nabla V=0$ on $S_\gamma$ becomes:
$$ 0 = \frac{1}{2} \big(\partial_i(V\circ \phi_\gamma) \cdot\partial_j\phi_\gamma
+ \partial_j(V\circ \phi_\gamma) \cdot\partial_i\phi_\gamma\big) =
~ \mbox{sym}[\nabla (v_1, v_2) + \nabla v_3\otimes \nabla v_0]_{ij}. $$
We also see that $v_3$ can be completed by $(v_1, v_2)$
to $V\in\mathcal{V}_1(S_h)$ as in (\ref{VSh}) only if:
\begin{equation}\label{VShv3}
\mbox{cof}\nabla^2v_0 : \nabla^2v_3 = 0,
\end{equation} the latter being also a sufficient condition when $\Omega$ is simply connected.    
This follows from:
\begin{equation*}
\begin{split}
\mbox{curl}^T&\mbox{curl}\Big(\mbox{sym}(\nabla v_3\otimes\nabla v_0)\Big)  = 
\mbox{curl}^T\mbox{curl}\Big(\nabla v_3\otimes\nabla v_0\Big) \\ & =
\partial_{22}(\partial_1v_3\cdot \partial_1v_0) 
+ \partial_{11}(\partial_2v_3\cdot\partial_2v_0)
- \partial_{12}(\partial_1v_3\cdot\partial_2v_0+\partial_2v_3\cdot\partial_1v_0)\\
& = -\big(\partial_{11}v_3\cdot\partial_{22}v_0 + \partial_{22}v_3\cdot\partial_{11}v_0 
- 2\partial_{12}v_3\cdot\partial_{12}v_0\big)
= - \mbox{cof} \nabla^2v_0:\nabla^2v_3.
\end{split}
\end{equation*} 
Hence, the admissible out-of-plane displacements $v_3$ 
relevant in (\ref{vonK}), must obey for the least the constraint (\ref{VShv3}), which 
appears in the 2-scale limiting theory 
(\ref{energy-shallow<1}) as constraint (\ref{constr2}). This is  in contrast
with the unconstrained 2-scale limiting theory (\ref{vonKnew1}) developed in the next section. 
\end{remark}

\begin{remark}
To put the last two results in another context, we draw the reader's
attention to the forthcoming paper \cite{LMP-new}, where we analyze the
$\Gamma$-limit of the shallow shell energies $\frac {1}{h^{2\alpha+2}}
I^{h^\alpha,h}$ on shells with curvature of order  $h^\alpha$. This energy
scaling is produced by forces of appropriate magnitude or by
prestrains of a different   
order than those considered in the present paper. Our main result in
\cite{LMP-new} concerns the case $\alpha<1$, where we can establish
that   in the special case $\det \nabla^2 v_0 \equiv c_0>0$, the
$\Gamma$-limit is a linearized Kirchhoff model with a Monge-Amp\`ere  
curvature constraint: 
\begin{equation}\label{constr}
\det \nabla^2 v= \det \nabla^2 v_0
\end{equation}
on the admissible out-of-plane displacements $v\in W^{2,2}(\Omega)$. The
constraint (\ref{constr2}) can be interpreted as a linearization
of  (\ref{constr}), thereby highlighting the relationship between the two models for elliptic shallow shells. 
\end{remark}

\section{The generalized Donnell-Mushtari-Vlasov model for a prestrained shallow shell: $\alpha=1$} \label{sec3}

When the parameter $\alpha =1,$  i.e. the curvature of the mid-surface co-varies with the thickness, so that $\gamma=h$. 
For small $h$,  the growth tensors on $(S_h)^h$ are then defined by (\ref{qh-gamma}) and 
the corresponding metric $p^h=(q^h)^Tq^h$ is given by:
$$p^h(\phi_h(x) + x_3\vec n^h(x)) = \mbox{Id} + 2h^2\mbox{sym }\epsilon_g(x) +
2hx_3\mbox{sym }\kappa_g(x) +\mathcal{O}(h^3).$$
Let $v^h=u^h\circ\tilde\phi_h\in W^{1,2}(\Omega^h,\mathbb{R}^3)$, 
via diffeomorphisms $\tilde\phi_h$ in (\ref{kl-gamma}).
By this simple change of variables, we see that:
\begin{equation*}
\begin{split}
I^{h,h}(u^h) & = \frac{1}{h}\int_{(S_h)^h} W(\nabla u^h(q^h)^{-1}) = 
\frac{1}{h}\int_{\Omega^h} W\Big((\nabla
v^h)(\nabla\tilde\phi_h)^{-1} (q^h\circ \tilde\phi_h)^{-1}\Big) \cdot
\det\nabla\tilde\phi_h~\mbox{d}(x, x_3) \\ & = \frac{1}{h}\int_{\Omega^h} W\Big((\nabla
v^h) (b^h)^{-1}\Big) \cdot \det\nabla\tilde\phi_h~\mbox{d}(x, x_3), 
\end{split}
\end{equation*}
where: 
$$b^h = (q^h\circ \tilde \phi_h)\nabla \tilde\phi_h.$$ 
In order to understand the structure of $b^h$ we need the following result: 

\begin{lemma}\label{lemik}
The pull-back of the metric $p^h$ through $\tilde\phi_h$ satisfies:
\begin{equation*}
\begin{split}
\forall (x, x_3)\in \Omega^h\qquad g^h(x, x_3)  & =  (\nabla\tilde\phi_h)^T (p^h\circ \tilde\phi_h) 
(\nabla \tilde\phi_h) \\ & = \mathrm{Id} +
h^2\Big(2\mathrm{sym}~\epsilon_g(x) + 
(\nabla v_0(x)\otimes\nabla v_0(x))^\ast \Big) \\ & \qquad
+ 2hx_3\Big(\mathrm{sym}~\kappa_g(x) - (\nabla^2 v_0(x))^\ast\Big) + \mathcal{O}(h^3),
\end{split}
\end{equation*}
where $F^\ast\in\mathbb{R}^{3\times 3}$ denotes the matrix whose only
  non-zero entries are in its $2\times 2$ principal minor given by $F\in\mathbb{R}^{2\times 2}$.
\end{lemma}
\begin{proof}
By a direct calculation, we obtain:
\begin{equation*}
\begin{split}
\partial_1\tilde\phi_h & = \big(1-x_3h\partial_{11}^2v_0,
-x_3h\partial_{12}^2v_0, h \partial_{1}v_0\big) + \mathcal{O}(h^3),\\
\partial_2\tilde\phi_h & = \big(-x_3h\partial_{12}^2v_0,
1-x_3h\partial_{22}^2v_0, h \partial_{2}v_0\big) + \mathcal{O}(h^3),\\
\partial_3\tilde\phi_h & = \vec n^h = \big(-h\partial_{1}v_0,
-h\partial_{2}v_0, 1-\frac{1}{2}h^2 |\nabla v_0|^2\big) + \mathcal{O}(h^3).
\end{split}
\end{equation*}
Hence:
\begin{equation*}
\begin{split}
&(\nabla\tilde\phi_h)^T (\nabla\tilde\phi_h) = \mbox{Id}_3 -2x_3h
(\nabla^2v_0)^\ast + h^2 (\nabla v_0\otimes \nabla v_0)^\ast + \mathcal{O}(h^3)\\
&(\nabla\tilde\phi_h)^T \big(2h^2\mbox{sym }\epsilon_g + 2hx_3\mbox{sym
}\kappa_g\big)(\nabla\tilde\phi_h)
= 2h^2\mbox{sym }\epsilon_g + 2hx_3\mbox{sym }\kappa_g+ \mathcal{O}(h^3),
\end{split}
\end{equation*}
in view of $\nabla\tilde\phi_h = \mbox{Id}_3 + \mathcal{O}(h)$, and the result follows.
\end{proof}

\medskip

Note that: $ (b^h)^T b^h = g^h $
and therefore by the polar decomposition of matrices:
$$ b^h = R(x,x_3) a^h \qquad \mbox{on } \Omega^h $$ 
for some $R(x,x_3)\in SO(3)$ and the symmetric growth tensor $a^h$ 
given by:
\begin{equation}\label{ahnew}
a^h = \sqrt{g^h} = \mathrm{Id} + h^2\Big(\mbox{sym }\epsilon_g + \frac{1}{2}
(\nabla v_0\otimes\nabla v_0)^\ast \Big)
+ hx_3\Big(\mbox{sym }\kappa_g - (\nabla^2 v_0)^\ast\Big) + \mathcal{O}(h^3).
\end{equation} 
For isotropic $W$ it directly follows that:
\begin{equation}\label{changev}
\begin{split}
I^{h,h}(u^h) & = \frac{1}{h}\int_{\Omega^h} W\Big((\nabla
v^h) (a^h)^{-1}R(x)^{-1}\Big) \cdot
\det\nabla\tilde\phi_h~\mbox{d}(x, x_3) \\ & 
=  \frac{1}{h}\int_{\Omega^h} W\Big((\nabla v^h) (a^h)^{-1}\Big)\cdot (1+ \mathcal{O}(h))~\mbox{d}(x, x_3).
\end{split}
\end{equation}
Heuristically, modulo the change of variable $\tilde \phi_h$ the problem reduces then 
to the study of deformations of the 
flat thin film $\Omega^h$ with the prestrain $a^h$. Indeed, by exactly the same analysis as in \cite{lemapa}
Theorems 1.2 and 1.3, we obtain in the general (not necessarily
isotropic) case, the following result:

\begin{theorem}\label{thnew}
Assume that $u^h\in W^{1,2}((S_h)^h,\mathbb{R}^3)$ satisfies
$I^{h,h}(u^h)\leq Ch^4$.
Then there exists proper rotations $\bar R^h\in SO(3)$ and
translations $c^h\in\mathbb{R}^3$ such that for the normalized
deformations:
\begin{equation*}\label{resc}
y^h(x,t) = (\bar R^h)^T (u^h\circ \tilde \phi_h)(x, ht) - c^h : \Omega^1\longrightarrow\mathbb{R}^3
\end{equation*}
defined by means of (\ref{kl-gamma}) on the common domain
$\Omega^1=\Omega\times (-1/2, 1/2)$ the following holds:
\begin{itemize}
\item[(i)] $y^h(x,t)$ converge in $W^{1,2}(\Omega^1,\mathbb{R}^3)$ to
  $x$.
\item[(ii)] The scaled displacements $V^h(x) = h^{-1}
  \fint_{-1/2}^{1/2}y^h(x,t) - x ~\mathrm{d}t$ converge (up to a
  subsequence) in $W^{1,2}(\Omega,\mathbb{R}^3)$ to the vector field
  of the form $(0,0,v)^T$ and $v\in W^{2,2}(\Omega, \mathbb{R})$.
\item[(iii)] The scaled in-plane displacements $h^{-1}V_{tan}^h$
  converge (up to a subsequence) weakly in $W^{1,2}$ to $w\in
  W^{1,2}(\Omega,\mathbb{R}^2)$.
\item[(iv)] Moreover: $\liminf_{h\to 0} h^{-4} I^{h,h}(u^h) \geq
  \mathcal{I}_4^1 (w,v)$ where:
\begin{equation}\label{vonKnew1}
\begin{split}
\mathcal{I}_4^1(w,v)= 
\frac{1}{2} 
\int_\Omega \mathcal{Q}_2&\left(\mathrm{sym }\nabla w 
+\frac{1}{2}\nabla v\otimes \nabla v - \frac{1}{2}\nabla
v_0\otimes\nabla v_0
- (\mathrm{sym}~\epsilon_g)_{tan}\right)\\
&\qquad\qquad
 + \frac{1}{24} \int_\Omega \mathcal{Q}_2\Big(\nabla^2 v - \nabla^2v_0
+ (\mathrm{sym}~\kappa_g)_{tan}\Big).
\end{split}
\end{equation}
\end{itemize}
\end{theorem}

In the same manner, applying the proof of Theorem 1.4 of \cite{lemapa} 
to (\ref{changev}), yields:

\begin{theorem}\label{recsec_sth}
For every $v\in W^{2,2}(\Omega,\mathbb{R})$ and $w\in
W^{1,2}(\Omega,\mathbb{R}^2)$, there exists a sequence of deformations
$u^h\in W^{1,2}((S_h)^h,\mathbb{R}^3)$ such that:
\begin{itemize}
\item[(i)] The sequence $y^h(x,t) = u^h(x+h v_0(x)e_3 + ht\vec
  n^h(x))$ converges in $W^{1,2}(\Omega^1,\mathbb{R}^3)$ to $x$.
\item[(ii)] The displacements $V^h$ as in (ii) Theorem \ref{thnew}
converge in $W^{1,2}$ to $(0,0,v)$.
\item[(iii)] The in-plane displacements $h^{-1}V^h_{tan}$ converge in
  $W^{1,2}$ to $w$.
\item[(iv)] $\lim_{h\to 0} h^{-4}I^{h,h}(u^h) = \mathcal{I}_{g, v_0}(w,v).$
\end{itemize}
\end{theorem}

\section{The prestrained plate model and the Euler-Lagrange equations: $\alpha >1$}

When the parameter $\alpha > 1,$ we calculate the pull-back of the induced metric $p^h = (q^h)^Tq^h$, to
the flat plate $\Omega^h$, via the change of variable $\tilde \phi_\gamma$ as in (\ref{kl-gamma}).
Just as in Lemma \ref{lemik}, we obtain: 
\begin{equation}\label{metric-alpha}
\begin{split}
g^h = (\tilde \phi_{h^\alpha})^\ast p^h= \mathrm{Id}_3 & +  h^{2\alpha}
(\nabla v_0\otimes\nabla v_0)^\ast - 2h^\alpha x_3 (\nabla^2 v_0)^\ast 
\\ & +  2h^2\mbox{sym }\epsilon_g +  2hx_3 \mbox{sym }\kappa_g + \mathcal{O}(h^3).
\end{split}
\end{equation} 
It is therefore clear that the prestrain terms $(\epsilon_g, \kappa_g)$ take over the effect of shallowness  
and hence the limiting theory in the scaling regime $h^4$ is that
derived in \cite{lemapa}, coinciding with results of Theorem \ref{thnew} and Theorem \ref{shallow<1-liminf} 
for the case $v_0=0$ and with the results of Theorem \ref{thmainuno} for $S\subset \R^2$:
\begin{equation}\label{vonKnew10}
\begin{split}
\forall v\in W^{2,2}(\Omega, \mathbb{R}) \quad \forall &w\in W^{1,2}(\Omega, \mathbb{R}^2) \\
\mathcal{I}_4^0(w,v)=  \frac{1}{2} 
\int_\Omega \mathcal{Q}_2&\left(\mathrm{sym }\nabla w 
+\frac{1}{2}\nabla v\otimes \nabla v 
- (\mathrm{sym}~\epsilon_g)_{tan}\right)\\
&\qquad\qquad \qquad\qquad \qquad
 + \frac{1}{24} \int_\Omega \mathcal{Q}_2\Big(\nabla^2 v 
+ (\mathrm{sym}~\kappa_g)_{tan}\Big).
\end{split}
\end{equation}
Indeed, consider the prestrained von K\'arm\'an   shell model $\mathcal{I}_4$ discussed in Section \ref{sec2} 
for a degenerate situation $S\subset \R^2$. The term $B- \frac{1}{2} (A^2)_{tan}$ 
reduces to: $\frac{1}{2} \left(\nabla w+ (\nabla w)^T + \nabla v \otimes \nabla  v\right)$,
where $w$ and $v=V^3$ are respectively the in-plane  and the
out-of-plane displacements of $S$. The term $(\nabla(A\vec n) -  A\Pi)_{tan}
$ reduces also to: $-\nabla^2 v$. Therefore, when $S\subset \R^2$,  $\mathcal{I}_4$ coincides with the model $\mathcal I^0_4$ and with the models 
$\mathcal {I}^\infty_4$ and $\mathcal{I}^1_4$ in the degenerate case $v_0=0$. 

\begin{remark}
We point out a qualitative difference between the out-of-plane displacements $v$
in the argument of $\mathcal{I}^0_4$ and $\mathcal{I}^1_4$ and those  appearing 
as the arguments of $\mathcal{I}^\infty_4$. 
The former are the net lowest order out of plane displacements of the limit deformations which are of order $h$,
as suggested by Theorem \ref{thnew} (ii), but, according to Theorem \ref{shallow<1-liminf} (ii), when $\alpha<1$, 
the latter are the second highest order term of the expansion of the deformation after $h^\alpha v_0$. 
Hence, one should replace $v$ in 
(\ref{vonKnew1}) or (\ref{new_Karman}) through a change of variables by $v+ h^{\alpha-1} v_0$ in order to 
quantitatively compare this model with the variational model $\mathcal{I}^\infty_4$ 
in (\ref{energy-shallow<1}).
\end{remark}

\medskip

As shown in \cite{lemapa}, under the assumption of $W$
being isotropic, 
the Euler-Lagrange equations of $\mathcal{I}_4$ under this degeneracy condition 
(or equivalently the Euler-Lagrange equations of $\mathcal{I}^0_4$) 
can be then written in terms of the displacement $v$ and the Airy stress potential $\Phi$: 
\begin{equation}\label{old_Karman}
\left \{ \begin{split}
\Delta^2\Phi & = -Y(\det \nabla^2 v + \lambda_g)\\
Z\Delta^2v &= [v,\Phi] - Z\Omega_g~,
\end{split}\right.
\end{equation} 
where $Y$ is the Young modulus, $Z$  the bending stiffness,
$\nu$  the Poisson ratio (given in terms of  the Lam\'e constants
$\mu$ and $\lambda$),  and :
\begin{equation}\label{lambdag} 
\begin{split}
\lambda_g & = \mbox{curl}^T\mbox{curl }(\epsilon_g)_{2\times 2}
= \partial_{22}(\epsilon_g)_{11} +  \partial_{11}(\epsilon_g)_{22}
- \partial_{12}\Big((\epsilon_g)_{12}+ (\epsilon_g)_{21}\Big),\\
\Omega_g & = \mbox{div}^T\mbox{div }\Big((\kappa_g)_{2\times 2}
+\nu\mbox{ cof }(\kappa_g)_{2\times 2}\Big) \\
& = \partial_{11} \Big((\kappa_g)_{11}+ \nu(\kappa_g)_{22}\Big) +
\partial_{22}\Big((\kappa_g)_{22}+ \nu(\kappa_g)_{11}\Big) +
(1- \nu)\partial_{12}\Big((\kappa_g)_{12}+ (\kappa_g)_{21}\Big).
\end{split}
\end{equation}
Equations (\ref{old_Karman}) are based on a thermoelastic analogy to
growth \cite{Mansfield, Maha} and can also be derived using a formal
perturbation theory \cite{BenAmar}.  
 
\medskip

On the other hand, the following system was introduced in \cite{Maha2}, 
as a mathematical model of blooming activated by differential lateral
growth from an initial non-zero transverse displacement field $v_0$:
\begin{equation}\label{new_Karman}
\left \{ \begin{split}
\Delta^2\Phi & = -Y(\det \nabla^2 v -\det\nabla^2 v_0+ \lambda_g)\\
Z(\Delta^2v-\Delta^2 v_0) &= [v,\Phi] - Z\Omega_g~,
\end{split}\right.
\end{equation} 
A similar calculation as in \cite{lemapa} then shows that 
(\ref{new_Karman}) can be viewed as the Euler Lagrange equations corresponding to the energy functional 
$\mathcal{I}^1_4$.  We will now show that (\ref{new_Karman}) can be 
directly derived from the equations (\ref{old_Karman}).
 
\begin{proposition}\label{derive}
The system (\ref{new_Karman}) can be
derived from the equations (\ref{old_Karman}) by pulling back the
prestrain tensors $\epsilon_g$ and $\kappa_g$ from a sequence of
shallow shells $(S_h)^h$ generated by the vanishing out-of-plane displacements $hv_0$.
\end{proposition}
 
\begin{proof} 
By Lemma \ref{lemik} we see that the growth tensor on
$\Omega^h$ is given by (\ref{ahnew}). Applying (\ref{lambdag}) to the modified strain and curvature in
$a^h$, to the leading order, we obtain:
\begin{equation*}
\begin{split}
\lambda_g(v_0) & = \mbox{curl}^T\mbox{curl }\Big((\mbox{sym }\epsilon_g)_{tan} +
\frac{1}{2}\nabla v_0\otimes \nabla v_0\Big) 
=  \lambda_g + \det\nabla^2v_0\\
\Omega_g(v_0) & = \mbox{div}^T\mbox{div }\Big(((\mbox{sym }\kappa_g)_{tan}-\nabla^2v_0)
+\nu\mbox{ cof }((\mbox{sym }\kappa_g)_{tan}-\nabla^2v_0)\Big) \\
& = \Omega_g -\Delta^2v_0,
\end{split}
\end{equation*}
where the last equality follows from $\mbox{div }\mbox{cof }\nabla^2v_0=0$.
Consequently, (\ref{old_Karman}) for the growth tensor 
(\ref{ahnew}) becomes exactly (\ref{new_Karman}).
\end{proof}

\section{The energy scaling}\label{sec5}

A straightforward consequence of our results is the following assertion about the scaling of the infimum  
elastic energies of the thin prestrained shallow shells in the von K\'arm\'an  regime (\ref{qh-gamma}). 
\begin{theorem} 
Let $\alpha> 0$ and let the sequence of thin  shells $(S_\gamma)^h$ be given as
in (\ref{shallow-gamma}) with the elastic energies of deformations $I^{\gamma,h}$ as in  
(\ref{IhW-gamma}). Assume that:
\begin{equation}\label{knot0}
\mathrm{curl} \, (\mathrm{sym}~\kappa_g)_{tan} \not\equiv 0
\quad \mbox { in } \Omega.
\end{equation} 
Then, there exists constants $c, C>0$ for which: 
\begin{equation}\label{teho}
\displaystyle  \forall 0<h\ll 1 \qquad 
c \le \inf_{u\in W^{1,2}((S_{h^\alpha})^h, {\mathbb R}^3)} \frac{1}{h^4} I^{h^\alpha, h}(u) \le C. 
\end{equation}
\end{theorem}
Indeed, the condition $\mbox{curl}(\mbox{sym }\kappa_g)_{tan}\equiv 0$
is equivalent to $(\mbox{sym }\kappa_g)_{tan} =
\nabla^2 v$, for some $v:\Omega\rightarrow \mathbb{R}$. If not 
satisfied, the bending term in (\ref{energy-shallow<1}) is always
positive, yielding the lower bound in (\ref{teho}). The existence of a recovery
sequence in Theorem \ref{shallow<1-limsup} and Theorem \ref{recsec_sth} and \cite{lemapa}
implies the upper bound.

\begin{remark}
The incompatibility condition (\ref{knot0}) can be relaxed depending
on the specific value of $\alpha$, and the assumed energy level,
see e.g. \cite{lemapa} for a more involved scaling analysis when
$\alpha>1$. Heuristically, conditions of similar type imply that the Riemann curvature tensor of the 
induced metric $p^h$ is non-zero and hence, in view of \cite[Theorem 2.2]{lepa_noneuclid},  
they guarantee the positivity of the infimum of $I^{\gamma,h}$. In a
further step we observe that, when $p^h$  is close to be  flat,
the scaling regime depends on the magnitude  of the first non-zero
term of the expansion of its curvature tensor. 
Note also that when $\alpha<1$, the first two non-zero terms after
identity in (\ref{metric-alpha}) 
have no bearing on the first non-zero terms in the expansion of the curvature.  
Analogously, the induced prestrains $\kappa_g'= \nabla^2 v_0$ and $\epsilon_g'=
\frac 12  (\nabla^2 v_0 \otimes \nabla^2 v_0)$ corresponding to the
scalings $h^\alpha$ and $h^{2\alpha}$ do not satisfy neither conditions (1.13) nor
(1.14) of \cite{lemapa}. Therefore the energy infimum must
naturally fall below $h^4$, i.e. in  the regime $h^{2\alpha+2}$.  
\end{remark} 

\section{Discussion}

Our analysis has rigorously derived a general theory of shells 
with residual strain arising from relative growth, inhomogeneous swelling, plasticity etc. 
In fact, there are many such theories; each is a consequence of the
scalings of shell's curvature relative to the magnitude of the strain incompatibility 
induced by curvature growth tensors.  Indeed, for any exponent $\alpha\geq 0$ we have considered the
following energies of deformations on weakly prestrained shallow shells:
\begin{equation*}
I^{h}(u) = \frac{1}{h}\int_{(S_{h^\alpha})^h} W((\nabla u) (q^h)^{-1}) \qquad
\forall u\in W^{1,2}((S_{h^\alpha})^h, \mathbb{R}^3),
\end{equation*}
with the growth tensor $q^h$ given by (\ref{qh-gamma}) on thin shells of
the form (\ref{shallow-gamma}) around the mid-surface:
$$S_{h^\alpha} = \phi_{h^\alpha} (\Omega), \qquad \phi_{h^\alpha} (x)=(x, h^\alpha v_0(x)), \qquad
v_0\in\mathcal{C}^{1,1} (\bar\Omega,\mathbb{R}).$$
We have established that independent of the value of $\alpha$, the scaling
for the infimum of the energy is always determined  
by the prestrain and is of order $h^4$ under our current assumption (\ref{knot0}).

When $\alpha>1$, the prestrain  overwhelms the role of shallowness so that the limiting theory is the one
derived in \cite{lemapa}, coinciding with results of Theorem \ref{thnew}
for the case $v_0=0$ and yielding the Euler-Lagrange equations (\ref{old_Karman}).
When $\alpha=1$ one recovers the recently postulated model \cite{Maha2},
discussed in the present paper. For the case $0<\alpha<1$, the limiting theory reduces to a new
constrained theory and can be viewed as a plate theory where the
non-trivial geometric structure of the shallow shell is inherited by
the plate, or equivalently it can be considered as the natural  
limit of the generalized von K\'arm\'an theories (\ref{vonK}) on the
shallow midsurface $S_\gamma$ as $\gamma\to 0$. This may be contrasted with a similar problem considered 
by the authors in \cite{LMP-new},  where the $\Gamma$-limit is discussed under energy regime of order
$h^{2\alpha+2}$. This order is compatible with the case where the role of shallowness is affected by relative scaled magnitude of the body forces or
prestrains, so that the choice of $\alpha$ has a bearing on the limiting model. Our analysis in the present
paper and in \cite{LMP-new} is this thus the beginning of an exploration from a vast possible scenarios. 
 
A natural generalization of our results would be to allow for different scaling regimes 
for the growth tensors in search of other possible limiting theories. 
Overall, there are three independent parameters: one associated with scaling of the shallowness, 
and two incompatible strains characterized in terms of their dependence on the thickness $h$
in the form  $h^\alpha$. The resulting theories depend on the choice of scalings for 
these three parameters. Thus, there is no single {\it correct} model in general, but of course
when dealing with a concrete situation, a choice of particular scalings 
for the relative magnitude of the thickness, the shallowness and the differential growth 
determines the effective theory, as we have shown here.

\bigskip

\noindent{\bf Acknowledgments.} 
This project is based upon work supported by, among others, the National Science
Foundation. M.L. is partially supported by the NSF grants DMS-0707275 and DMS-0846996, 
and by the Polish MN grant N N201 547438. L.M.  
is supported by the MacArthur Foundation, M.R.P. is partially supported by 
the NSF grants DMS-0907844 and DMS-1210258. 

\newpage

\begin{center}
{\bf \Large  Appendices}
\end{center}

Here we provide a proof of Theorem \ref{shallow<1-liminf} and Theorem \ref{shallow<1-limsup}.
We first derive  bounds on families of vector mappings $\{u^h\}_{h>0}$,
defined on $(S_\gamma)^h$ as in (\ref{shallow-alpha-gamma}) and (\ref{shallow-gamma}),  
under the assumption of smallness on their energy and when $\gamma=h^\alpha$ 
for the scaling regime $0<\alpha<1$. In what follows, by $C$ we denote an arbitrary positive constant, 
depending on $v_0$, but not on $h$ or the vector mapping under consideration.
In all proofs, the convergences are understood up to a subsequence, 
unless stated otherwise.

\appendix

\section{Proof of Theorem \ref{shallow<1-liminf}}

Let  a sequence of deformations $u^h\in W^{1,2}((S_\gamma)^h, \mathbb{R}^3)$ satisfy:
 $$\displaystyle \frac{1}{h} \int_{(S_\gamma)^h}W((\nabla u^h)(q^h)^{-1})~\mbox{d}z \leq C h^4,$$
where, recalling the definition of $\tilde\phi_h:\Omega\times (-h/2,
h/2)\to (S_\gamma)^h$ in (\ref{kl-gamma}), we have:
$$ q^h \circ \tilde \phi_\gamma = \mbox{Id} + h^2 \epsilon_g + hx_3
\kappa_g.$$ 
The following approximation results can be obtained by combining arguments 
in \cite{lemopa7} and \cite{lemapa}, in view of the seminal work \cite{FJMhier}. 

\begin{lemma} \label{appr} 
(Parallel to \cite[Lemma 3.1]{lemopa7}, \cite[Theorem 1.6]{lemapa}.)
There exists a matrix field $R^h\in W^{1,2}(S_\gamma,
\mathbb{R}^{3\times 3})$ with values in $SO(3)$, and a matrix $Q^h\in SO(3)$, 
such that:
\begin{itemize}
\item[(i)] $\displaystyle \frac 1h \|\nabla u^h - R^h\|^2_{L^2(S_\gamma^h)} \leq Ch^4,$
\item[(ii)] $\|\nabla R^h\|_{L^2(S_\gamma)} \leq Ch,$
\item[(iii)] $\|(Q^h)^TR^h - \mathrm{Id}\|_{L^p(S_\gamma)} \leq Ch,~$ for all $p\in [1,\infty)$.
\end{itemize}
\end{lemma}
 
\begin{lemma} \label{lem3.2} (Parallel to \cite[Lemma 3.2]{lemopa7}.) 
Let $R^h, Q^h$ be as in Lemma \ref{appr}. There holds, or all $p\in [1,\infty)$:
\begin{itemize}
\item[(i)] $\displaystyle \lim_{h\to 0}~ (Q^{h})^TR^{h} \circ \phi_\gamma =\mathrm{Id}$, in $W^{1,2}(\Omega)$ and in
$L^p(\Omega)$.
\end{itemize}
Moreover, there exists a $W^{1,2}$ skew-symmetric field $A:S\longrightarrow so(3)$, such that:
\begin{itemize}
\item[(ii)] $\displaystyle \lim_{h\to 0}~ \frac{1}{h} \left( (Q^{h})^TR^{h} 
- \mathrm{Id}\right) \circ \phi_\gamma = A$, weakly in $W^{1,2}(\Omega)$ and (strongly) in $L^p(\Omega)$.
\item[(iii)] $\displaystyle \lim_{h\to 0}~ \frac{1}{h^2} \mathrm{sym }\left( (Q^{h})^TR^{h} 
- \mathrm{Id}\right) \circ \phi_\gamma = \frac{1}{2}A^2$, in $L^p(\Omega)$.
\end{itemize}
\end{lemma}
\begin{proof}
The convergences in (i) follow from Lemma \ref{appr}.
To prove (ii), notice that:
$$A^h = \frac{1}{h} \left((Q^h)^T R^h - \mathrm{Id}\right) \circ \phi_\gamma$$
is bounded in $W^{1,2}(\Omega)$ and so it has a weakly converging
subsequence, as $h\to 0$. Consequently, convergence is strong in $L^p(\Omega)$. One has:
\begin{equation}\label{symAh}
A^h + (A^h)^T = \frac{1}{h} \left((Q^h)^T R^h +  (R^h)^TQ^h - 2\mathrm{Id} \right) 
= - {h}(A^h)^TA^h. 
\end{equation} 
The latter converges to $0$ in $L^p(\Omega)$, and therefore the limit matrix field $A$ is skew-symmetric.
The above equality proves as well that:
$$\lim_{h\to 0}  \frac{1}{h} \mbox{ sym } A^h = \frac{1}{2} A^2$$
in $L^p(\Omega)$, which implies (iii).
\end{proof}

Consider (and compare with Theorem \ref{shallow<1-liminf}) the rescaling:
\begin{equation*}
y^h(x + t\vec n^\gamma(x)) = u^h\left(x + t{h}/{h_0}\vec n^\gamma(x)\right) 
\qquad \forall x\in S_\gamma\quad \forall t\in (-h_0/2, h_0/2),
\end{equation*}
so that $y^h\in W^{1,2}((S_\gamma)^{h_0}, \mathbb{R}^3)$. Also, define:
$\nabla_h y^h(x + t\vec n^\gamma(x)) = \nabla u^h\left(x + t{h}/{h_0}\vec n^\gamma(x)\right)$.
In what follows, $\Pi_\gamma=\nabla{\vec n}^\gamma$ denotes the second fundamental form of $S_\gamma$.
By a straightforward calculation we obtain:

\begin{proposition}\label{formule} (Parallel to \cite[Proposition 3.3]{lemopa7}.)
For each $x\in S_\gamma$, $t\in (-h_0/2, h_0/2)$ and $\tau\in T_xS_\gamma$ there hold:
\begin{equation*}
\begin{split}
\partial_{\tau} y^h(x+t\vec n^\gamma) &= 
\nabla_h y^h\left(x + t\vec n^\gamma\right)\left(\mathrm{Id} + t{h}/{h_0}\Pi_\gamma(x)\right)
(\mathrm{Id} + t\Pi_\gamma(x))^{-1} \tau \\
{\partial_{\vec n^\gamma}} y^h(x+t\vec n^\gamma) &= \frac{h}{h_0} 
\nabla_h y^h\left(x + t\vec n^\gamma\right)\vec n^\gamma(x).
\end{split}
\end{equation*}
Moreover, for $I^h(y^h) = \frac{1}{h} \int_{(S_\gamma)^h}W((\nabla u^h)(q^h)^{-1})$ one has:
\begin{equation*}
\begin{split}
I^h(y^h) &= \frac{1}{h_0}\int_{(S_\gamma)^{h_0}} W(\nabla_h y^h (x+t\vec n^\gamma) (q^h)^{-1})\cdot 
\det \left[\left(\mathrm{Id} + t{h}/{h_0}\Pi_\gamma\right)
(\mathrm{Id} + t\Pi_\gamma)^{-1}\right]\\
&= \int_{S_\gamma} \fint_{-h_0/2}^{h_0/2} W(\nabla_h y^h (x+t\vec n^\gamma) (q^h)^{-1})\cdot 
\det \left[\mathrm{Id} + t{h}/{h_0}\Pi_\gamma(x)\right]~\mathrm{d}t~\mathrm{d}x.
\end{split}
\end{equation*}
\end{proposition}
Directly from Lemma \ref{appr} (i) and Lemma \ref{lem3.2} (ii) 
there follows:

\begin{proposition} \label{help}  (Parallel to \cite[Proposition 3.4]{lemopa7}.)
\begin{itemize}
\item[(i)] $\displaystyle \|\nabla_h y^h - R^h\|_{L^2((S_\gamma)^{h_0})}\leq Ch^2$. 
\item[(ii)] $\displaystyle \lim_{h\to 0} \frac{1}{h}\left((Q^h)^T\nabla_h y^h - \mathrm{Id}\right) \circ \tilde \phi_\gamma
= A$, in $L^{2}(\Omega^{h_0})$.
\end{itemize}
\end{proposition}

We consider the corrected by rigid motions deformations 
$\tilde y^h\in W^{1,2}((S_\gamma)^{h_0},\mathbb{R}^3)$ and averaged displacements
$V^h\in W^{1,2}(S_\gamma, \mathbb{R}^3)$:
\begin{equation*}
\tilde y^h = (Q^h)^T y^h - c^h,\qquad
V^h = V^h[\tilde y^h] 
= \frac{1}{h}\fint_{-h_0/2}^{h_0/2} \tilde y^h(x+t\vec n^\gamma) - x ~\mathrm{d}t,
\end{equation*}
where the constants $c^h$ are chosen so that $\fint_{\Omega} V^h \circ \phi_\gamma = 0$.

\begin{lemma} \label{lem3.4}  (Parallel to \cite[Lemma 3.5]{lemopa7}.) 
\begin{itemize}
\item[(i)] $\displaystyle \lim_{h\to 0} (\tilde y^{h}\circ \tilde \phi_\gamma  - \phi_\gamma) =0$ in $W^{1,2}(\Omega^{h_0}).$
\item[(ii)] $\displaystyle \lim_{h\to 0} (V^{h}\circ \phi_\gamma) = V$  in $W^{1,2}(\Omega)$.
\end{itemize} 
The vector field $V$ in (ii) has regularity $W^{2,2}(\Omega, \mathbb{R}^3)$ 
and it satisfies $\partial_\tau V (x) = A(x) \tau$
for all $\tau\in \mathbb{R}^2S$. The $W^{1,2}$ skew-symmetric matrix field
$A:S\longrightarrow so(3)$ is as in Lemma \ref{lem3.2}.
\end{lemma}
\begin{proof} {\bf 1.} 
In view of Proposition \ref{formule} and Proposition \ref{help} we have:
\begin{equation}\label{important}
\begin{split}
&\left\|\nabla_{tan}\tilde y^h - \left((Q^h)^T R^h\right)_{tan}\cdot 
(\mathrm{Id} + th/h_0 \Pi_\gamma^h) (\mathrm{Id} + t\Pi_\gamma)^{-1} 
\right\|_{L^2(S_\gamma^{h_0})} \leq Ch^2\\
&\left\|\partial_{\vec n^\gamma}\tilde y^h \right\|_{L^2((S_\gamma)^{h_0})} \leq
Ch\|\nabla_h y^h\|_{L^2(S_\gamma^{h_0})} \leq Ch.
\end{split}
\end{equation}
To prove convergence of $V^h \circ \phi_\gamma$, consider for $x\in S_\gamma$:
\begin{equation}\label{nabla_Vh}
\begin{split}
\nabla V^h(x) &= \frac{1}{h} \fint_{-h_0/2}^{h_0/2} \nabla_{tan}\tilde y^h(x+t\vec n^\gamma)
(\mathrm{Id} + t\Pi_\gamma) - \mathrm{Id} ~\mbox{d}t\\ 
&= \frac{1}{h} \fint_{-h_0/2}^{h_0/2} \Big(\nabla_{tan}\tilde y^h
- \left((Q^h)^T R^h\right)_{tan} (\mathrm{Id} + t\Pi_\gamma)^{-1}\Big) (\mathrm{Id} + t\Pi_\gamma) ~\mbox{d}t\\ 
&\quad + \frac{1}{h} \Big((Q^h)^T R^h (x) - \mathrm{Id}\Big)_{tan} = A^h \circ (\phi_\gamma)^{-1}+ {\mathcal O}(h).
\end{split}
\end{equation}
We also have: $ \nabla (V^h\circ \phi_\gamma) = (\nabla V^h \circ \phi_\gamma ) (\nabla \phi_\gamma),$ hence: 

\begin{equation}\label{nabla_Vh-local}
\begin{split}
\nabla (&V^h \circ \phi_\gamma)  = \Big [( \frac{1}{h} \fint_{-h_0/2}^{h_0/2} \nabla_{tan}\tilde y^h 
(\mathrm{Id} + t\Pi_\gamma) - \mathrm{Id} ~\mbox{d}t  )\circ  \phi_\gamma \Big] \nabla \phi_\gamma \\ 
&= \Big [( \frac{1}{h} \fint_{-h_0/2}^{h_0/2} \Big(\nabla_{tan}\tilde y^h
- \left((Q^h)^T R^h\right)_{tan} (\mathrm{Id} + t\Pi_\gamma)^{-1}\Big)
(\mathrm{Id} + t\Pi_\gamma) ~\mbox{d}t) \circ \phi_\gamma \Big ]  
\nabla \phi_\gamma \\  &\quad + \Big [( \frac{1}{h} \Big((Q^h)^T R^h
(x) - \mathrm{Id}\Big)_{tan})\circ \phi_\gamma \Big ] \nabla
\phi_\gamma  
= A^h \nabla \phi_\gamma + \mathcal O(h).
\end{split}
\end{equation}
Therefore, $\nabla (V^h\circ \phi_\gamma)$  converges to $A_{tan}$ in
$L^2(\Omega)$ and since $\fint_{\Omega} V^h \circ \phi_\gamma = 0$,  
we may use Poincar\'e inequality on $\Omega$ to deduce (ii).

{\bf 2.}
To prove (i), notice that by (\ref{important}) and Lemma \ref{lem3.2} we obtain the following 
convergences in $L^2(\Omega^{h_0})$:
\begin{equation*}
\begin{split}
&\lim_{h\to 0}\nabla (\tilde y^h \circ \tilde \phi_\gamma)  - \nabla \tilde \phi_\gamma
 =\lim_{h\to 0} \Big ( (\nabla \tilde y^h - \mbox{Id})\circ \tilde \phi_\gamma\Big ) \nabla \tilde \phi_\gamma  =0,\\
&\lim_{h\to 0}{\partial_3}(\tilde y^h \circ \tilde \phi_\gamma)=\lim_{h\to 0} (\partial_{\vec n^\gamma} \tilde y^h)\circ \tilde \phi_\gamma = 0.
\end{split}
\end{equation*}
Therefore $\nabla (\tilde y^h\circ \tilde \phi_\gamma ) -\nabla\phi_\gamma$ converges to $0$ in $L^2(\Omega^{h_0})$.
Since the sequence $\{V^h \circ \phi_\gamma\}$ is bounded in $L^2(\Omega)$, it also follows that:
\begin{equation}\label{conve}
\lim_{h\to 0} \left\|\int_{-h_0/2}^{h_0/2}\tilde y^h (\tilde
  \phi_\gamma) - x ~\mbox{d}t\right\|_{L^2(S_\gamma)} = 0.
\end{equation}
Now, let $g(x+t\vec n^\gamma) = |\mbox{det } (\mathrm{Id} +t\Pi_\gamma(x))|^{-1}$. Consider
the two terms in the right hand side of:
$$\|\tilde y^h - \pi\|_{L^2((S_\gamma)^{h_0})} \leq
\left\|(\tilde y^h - \pi) - \int_{(S_\gamma)^{h_0}}(\tilde y^h - \pi)\cdot \frac{g}{\scriptstyle \int_{(S_\gamma)^{h_0}} g}
\right\|_{L^2((S_\gamma)^{h_0})}
+ ~\left|\int_{S_\gamma^{h_0}}(\tilde y^h - \pi)\cdot \frac{g}{\scriptstyle\int_{(S_\gamma)^{h_0}} g}\right|.$$
The first term  can be bounded by means of the weighted Poincar\'e inequality, by
$\|\nabla (\tilde y^h -\pi)\|_{L^2(S_\gamma^{h_0})}$ and therefore it converges to $0$ as $h\to 0$.
The second term converges to $0$ as well, in view of (\ref{conve}) and:
$$\left|\int_{(S_\gamma)^{h_0}} (\tilde y^h - \pi)\cdot g\right| = 
\left|\int_{S_\gamma}\int_{-h_0/2}^{h_0/2} \tilde y^h - \pi ~\mbox{d}t~\mbox{d}x\right| \leq
C \left\|\int_{-h_0/2}^{h_0/2}\tilde y^h - \pi ~\mbox{d}t\right\|_{L^2(S_\gamma)}.$$
\end{proof}

\begin{proposition}\label{new-constraint}
We have: 
\begin{itemize}
\item[(i)] 
$ \displaystyle \lim_{h\to 0} \frac{1}{h^\alpha} \Big ( (\tilde y^h
\circ \tilde \phi_\gamma)  - x \Big ) = v_0 e_3$ 
in $W^{1,2}(\Omega, \mathbb{R}^3)$,
\item[(ii)] $\displaystyle \mathrm{cof}\,\,\nabla^2 v_0: \nabla^2 V^3
  = \mathrm{cof}\,\,\nabla^2 v_0: \nabla^2 v =0$ in $\Omega$.
\end{itemize}
\end{proposition}
\begin{proof} The statement (i) easily follows from  Lemma
  \ref{lem3.4} (i).  By (\ref{nabla_Vh-local}) and \eqref{symAh} we calculate: 
\begin{equation*} 
\begin{split} 
\forall i,j=1,2 \quad  & 2 \left\langle\partial_i \phi_\gamma, (\mbox{sym} \nabla V^h)  \partial_j \phi_\gamma\right\rangle 
= \left\langle \partial_i (V^h \circ \phi_\gamma), \partial_j
  \phi_\gamma \right\rangle + \left\langle\partial_j (V^h \circ \phi_\gamma), \partial_i \phi_\gamma\right\rangle
\\ & = \left\langle \partial_j \phi_\gamma,  A^h \partial_i
  \phi_\gamma\right\rangle + \left\langle\partial_i \phi_\gamma,  A^h \partial_j \phi_\gamma\right\rangle + 
\mathcal O(h) = \mathcal O(h), 
\end{split}
\end{equation*} 
and hence, denoting: $V^h\circ \phi_\gamma= (v_1^h, v_2^h, v_3^h)$, we get:
\begin{equation*} 
\begin{split} 
\mathcal O(h) = 2 \left\langle\partial_i \phi_\gamma, (\mbox{sym} \nabla V^h)  \partial_j \phi_\gamma\right\rangle & 
= \left\langle\partial_i (V^h \circ \phi_\gamma), \partial_j
  \phi_\gamma\right\rangle  + \left\langle \partial_j (V^h \circ
  \phi_\gamma), \partial_i \phi_\gamma\right\rangle
\\ & = 2\left\langle e_i, \big(\mbox{sym} \nabla(v^h_1, v^h_2) +
  h^\alpha \mbox{sym} (\nabla v^h_3 \otimes \nabla v_0)\big)  e_j\right\rangle.
\end{split}
\end{equation*} 
Dividing by $h^\alpha$ and passing to $0$ in $h$ we obtain: 
$$  \lim_{h\to 0}  \Big (\mbox{sym} (\nabla v^h_3 \otimes \nabla v_0)
+  \frac {1}{h^\alpha} \mbox{sym} \nabla(v^h_1, v^h_2) \Big)=0. $$ 

On the other hand $V^h \circ \phi_\gamma$ converges to $V$ in $\Omega$
and so $\mbox{sym} \nabla V=0$ imply that $(v^h_1, v^h_2)$ 
converge to a constant, while
$v^h_3$ converges to $V^3=v\in W^{2,2}(\Omega)$. Passing to the limit we obtain:
$$  \mbox{sym} \big(\nabla v \otimes \nabla v_0\big)  = - \lim_{h\to
  0} \frac {1}{h^\alpha} \mbox{sym} \nabla(v^h_1, v^h_2)  
= - \mbox{sym} \nabla \tilde w,  $$
for some $\tilde w \in W^{1,2}(\Omega)$ (we used Korn's inequality for
deducing the existence of $\tilde w$). By applying the operator
$\mbox{curl}^T\mbox{curl}$ on both sides, we conclude:
$$ \displaystyle \mbox{cof }\nabla^2 v_0: \nabla^2 v = 0,$$ 
as claimed in (ii).
\end{proof} 

We now need to study the following sequence of matrix fields on $(S_\gamma)^{h_0}$:
$$G^h = \frac{1}{{h}} \Big((R^h)^T \nabla_h y^h - \mathrm{Id}\Big) \circ \tilde \phi_\gamma.$$
In view of Proposition \ref{help} (i), the tensor $2\mbox{sym } G^h$ is the $h^2$ order term in the 
expansion of the nonlinear strain $(\nabla u^h)^T \nabla u^h$, at $\mbox{Id}$.

\begin{lemma} \label{lem3.6} (Parallel to \cite[Lemma 3.6]{lemopa7}.)
The sequence $\{G^h\} $ as above has a subsequence, converging weakly
in $L^2(\Omega^{h_0})$ to a matrix field $G$.  The tangential minor of $G$ is
affine in the $e_3$ direction.  More precisely:
$$ G(x,t)_{2\times 2} = G_0(x)_{2\times 2} - \frac{t}{h_0} \nabla^2
v(x), ~~\mbox{ with }~~ G_0(x) = \fint_{-h_0/2}^{h_0/2} G(x, t)~\mathrm{d}t.$$
\end{lemma}
\begin{proof}
{\bf 1.} The sequence $\{G^h\}$ is bounded in $L^2(\Omega^{h_0})$ by Proposition \ref{help} (i).  
Therefore it has a subsequence (which we do not relabel)
converging weakly to some $G$.
For a fixed $s>0$, consider now the sequence of vector fields
$f^{s,h}\in W^{1,2}((S_\gamma)^{h_0}, \mathbb{R}^3)$:
$$f^{s,h}(x+t\vec n^\gamma) = \frac{1}{sh^2} \Big[\Big(h_0\tilde y^h(x+(t+s)\vec n^\gamma) - h(x+(t+s)\vec n^\gamma)\Big)
- \Big(h_0\tilde y^h(x+t\vec n^\gamma) - h(x+t\vec n^\gamma)\Big)\Big]$$
We claim that $f^{s,h} \circ \tilde \phi_\gamma$ converges in
$L^2(\Omega^{h_0})$ to $Ae_3$. Indeed, by Proposition \ref{formule} one has:
\begin{equation*}
\begin{split}
f^{s,h}(x+t\vec n^\gamma) &= \frac{1}{{h^2}} 
\fint_t^{t+s} \Big(h_0\partial_{\vec n^\gamma} \tilde y^h(x+\sigma \vec n^\gamma) - h\vec n^\gamma\Big)
~\mbox{d}\sigma\\ &= \frac{1}{h} \fint_{t}^{t+s} 
\Big((Q^h)^T\nabla_h y^h (x+\sigma\vec n^\gamma) - \mathrm{Id}\Big)\vec n^\gamma~\mbox{d}\sigma,
\end{split}
\end{equation*}
and the convergence follows by Proposition \ref{help} (ii).

\medskip

{\bf 2.}
We claim that this convergence is actually weak in $W^{1,2}(\Omega^{h_0})$.
First, notice that the $x_3$ derivatives converge to $0$ in $L^2(\Omega^{h_0})$ 
by Proposition \ref{help} (ii):
\begin{equation*}
\partial_{3}(f^{s,h}\circ \tilde \phi_\gamma)   = \frac{1}{sh} ~\Big [ 
(Q^h)^T \Big(\nabla_h y^h(x+(t+s)\vec n^\gamma) - \nabla_h y^h(x+t\vec n^\gamma)\Big) \circ \tilde \phi_\gamma\Big ] \vec n^\gamma(x).
\end{equation*}
We now find the weak limit of the tangential gradients of $f^{s,h}$.
By Proposition \ref{formule}:
\begin{equation*}
\begin{split}
&\partial_{i} (f^{s,h} \circ \tilde \phi_\gamma) = \frac{1}{sh^2} ~
\Big [\Big(h_0\nabla \tilde y^h(x+(t+s)\vec n^\gamma)(\mathrm{Id} + (t+s)\Pi_\gamma)(\mathrm{Id} + t\Pi_\gamma)^{-1} \\
& \qquad\qquad\qquad\qquad\qquad
- h_0\nabla \tilde y^h(x+t\vec n^\gamma) - hs\Pi_\gamma (\mathrm{Id} + t\Pi_\gamma)^{-1}\Big) \Big] \circ \tilde \phi_\gamma  \partial_i \phi_\gamma\\
&=\frac{h_0}{sh^2} ~\Big [ 
(Q^h)^T \Big(\nabla_h y^h(x+(t+s)\vec n^\gamma) - \nabla_h y^h(x+t\vec n^\gamma)\Big)
(\mathrm{Id} + th/h_0\Pi_\gamma)(\mathrm{Id} + t\Pi_\gamma)^{-1}  \Big] \circ \tilde \phi_\gamma  \partial_i \phi_\gamma \\
&\quad + \frac{1}{sh}~\Big [ \Big((Q^h)^T\nabla_h y^h(x+(t+s)\vec n^\gamma) - \mathrm{Id}\Big)
s \Pi_\gamma(\mathrm{Id} + t\Pi_\gamma)^{-1} \Big] \circ \tilde \phi_\gamma  \partial_i \phi_\gamma .
\end{split}
\end{equation*}
By Proposition \ref{help} (ii), the following expression in the right hand side above:
$$\frac{1}{h}~\Big [ \Big((Q^h)^T\nabla_h y^h(x+(t+s)\vec n^\gamma) - \mathrm{Id}\Big)
\Pi_\gamma(\mathrm{Id} + t\Pi_\gamma)^{-1}] \Big] \circ \tilde \phi_\gamma $$
converges in $L^2(\Omega^{h_0})$ to $0$.
On the other hand, the first term in this expression 
converges weakly in $L^2(\Omega^{h_0})$ to:
$$\frac{h_0}{s} (G(x,t+s)) - G(x,t)),$$
by Lemma \ref{lem3.2} (i). This establishes the (weak) convergence of $f^{s,h}$ in $W^{1,2}(\Omega^{h_0})$.

\medskip

{\bf 3.} Equating the weak limits of tangential derivatives, we obtain:
\begin{equation*}
\begin{split}
\partial_i (A e_3) (x)  &=
\frac{h_0}{s}\Big(G(x,(t+s)) - G(x,t)\Big)e_i.
\end{split}
\end{equation*}
This proves the lemma.
\end{proof}

Finally, we have the following bound for convergence of the scaled energies $I^h$.

\begin{lemma}\label{liminf} (Parallel to \cite[Lemma 3.7]{lemopa7} and \cite[Theorem 1.3]{lemapa}.)
$$\liminf_{h\to 0} \frac{1}{h^4} I^h(y^h) \geq \frac{1}{2} \int_S \mathcal{Q}_2\left((\mathrm{sym }~ 
(G_0)_{2 \times 2} - (\epsilon_g)_{2\times 2}\right) + \frac{1}{24}
\int_S \mathcal{Q}_2\left(\nabla^2 v + (\kappa_g)_{2\times 2}
\right).$$ 
\end{lemma}

In view of Lemma \ref{lem3.4} and Lemma \ref{liminf}, it remains to
understand the structure of $G_0$.  
 
\begin{lemma}\label{lem3.9} (Parallel to \cite[Lemma 3.7]{lemopa7}.) 
Let $G_0$ be as in Lemma \ref{lem3.6}. 
Then we have the following convergence (up to a subsequence) weakly in $L^2(\Omega)$:
\begin{equation}\label{symG0tan}
\lim_{h\to 0} \frac{1}{h} \left\langle\partial_j \phi_\gamma, (
  [\mathrm{sym }~\nabla V^h] \circ \phi_\gamma)  \partial_i \phi_\gamma\right\rangle 
= \left\langle e_i, \left(\mathrm{sym }~G_0 + \frac{1}{2} A^2\right)_{2\times 2} e_j\right\rangle.
\end{equation}
\end{lemma}
\begin{proof}
We use the formula (\ref{nabla_Vh}) composed with $\phi_\gamma$ to calculate 
$\frac{1}{h}(\mathrm{sym }~ \nabla V^h) \circ \phi_\gamma$.
The last term in the right hand side gives:
$$\frac{1}{{h^2}} \mbox{sym} \left((Q^h)^T R^h - \mbox{Id}\right)_{tan} \circ \phi_\gamma
=  \frac{1}{h^2} \mbox{sym}\left((Q^h)^T R^h - \mbox{Id}\right)_{tan} \circ \phi_\gamma,$$
which converges in $L^2(\Omega)$ to $1/2 (A^2)_{tan}$ by Lemma \ref{lem3.2} (iii).
To treat the first term in the right hand side of (\ref{nabla_Vh}), notice that
for every $\tau\in T_xS_\gamma$:
\begin{equation*}
\begin{split}
&\langle \frac{1}{h^2} \Bigg[\fint_{-h_0/2}^{h_0/2} \nabla\tilde y^h(x+t\vec n^\gamma) (\mbox{Id} + t\Pi_\gamma)
- (Q^h)^T R^h(x) ~\mbox{d}t\Bigg] \circ \phi_\gamma, \tau\rangle \\
&\quad = \frac{1}{h^2}\langle (Q^h)^T \Bigg[\fint_{-h_0/2}^{h_0/2}\nabla_h y^h(x+t\vec n^\gamma) - R^h(x) 
~\mbox{d}t + \fint_{-h_0/2}^{h_0/2} t h/h_0 \nabla_h y^h \Pi_\gamma~\mbox{d}t\Bigg]  \circ \phi_\gamma, \tau\rangle\\
&\quad= \frac{1}{h^2}\langle (Q^h)^T R^h(x)
\Bigg[\fint_{-h_0/2}^{h_0/2}(R^h)^T\nabla_h y^h - \mbox{Id}~\mbox{d}t\Bigg]  \circ \phi_\gamma, \tau\rangle \\
&\qquad\qquad\qquad\qquad
+ \frac{h_0}{h}\langle\Big ( (Q^h)^T\Bigg[\fint_{-h_0/2}^{h_0/2} t \left(\nabla_h y^h- R^h\pi\right) 
~\mbox{d}t \Bigg]\Pi_\gamma(x) \Big )\circ \phi_\gamma,  \tau\rangle, 
\end{split}
\end{equation*}
where we used Proposition \ref{formule}.
Now, the second term in the right hand side above converges in $L^2(\Omega)$ to $0$, 
by Proposition \ref{help} (i).
Further, the matrix in the first term equals to:
$$((Q^h)^T R^h(x) \fint_{-h_0/2}^{h_0/2}G^h(x+t\vec n^\gamma)~\mbox{d}t ) \circ \phi_\gamma,$$
and by Lemma \ref{lem3.2} (i) and Lemma \ref{lem3.6}, it converges weakly 
in $L^2(\Omega)$ to $G_0$.
This completes the proof, in view of the fact that $\nabla \phi_\gamma$ converges to $(e_1, e_2)$.
\end{proof}

\medskip

\noindent {\bf Conclusion of the proof of Theorem
  \ref{shallow<1-liminf}}

It now remains to identify $B$ as a member of ${\mathcal B}_{v_0}$.
We have: 
\begin{equation}\label{strange} 
\frac 1h \langle \partial_j \phi_\gamma, ( [\mathrm{sym }~\nabla (V^h)] \circ \phi_\gamma) 
\partial_i \phi_\gamma \rangle = \frac 1h \langle e_i, \big(\mbox{sym}
\nabla(v^h_1, v^h_2) + h^\alpha \mbox{sym} (\nabla v^h_3 \otimes
\nabla v_0)\big) e_j\rangle. 
\end{equation} 
Therefore: 
$$ \mbox{sym} (G_0)_{2\times 2} = B  - \frac{A^2}{2}$$ 
where $B$ is given by the following limit: 
$$ B= \lim_{h\to 0} B^h=  \lim_{h\to 0}   \frac 1h 
\Big [\mbox{sym} \nabla(v^h_1, v^h_2) + h^\alpha \mbox{sym} (\nabla
v^h_3 \otimes \nabla v_0) \Big ] \in {\mathcal B_{v_0}}, $$ 
whose existence is assured by Lemma \ref{lem3.9}. \endproof

\section{Proof of Theorem \ref{shallow<1-limsup}}

Let $v$ and $B$ be given as in the statement of Theorem
\ref{shallow<1-limsup}.  Since $v$ satisfies the constraint
(\ref{constr2}) on a simply-connected $\Omega$, there exists a
displacement field $\tilde w\in W^{2,2}(\Omega, {\mathbb R}^2)$ 
for which: 
$$ \mbox{sym} \nabla \tilde w+ \mbox{sym} (\nabla v \otimes \nabla v_0) =0. $$
Let $V_\gamma= (\gamma \tilde w, v)\circ (\phi_\gamma)^{-1}$. 
Then it is straightforward as in \eqref{strange} to verify that $V_\gamma$ 
is a first order isometry of class $W^{2,2}$ on $S_\gamma$,
i.e. $\mbox{sym} \nabla V_\gamma=0$ on $S_\gamma$.  
Also, using the isomorphism ${\mathcal T}^\gamma$ in Lemma 4.1, 
we let $B_\gamma= ({\mathcal T}^\gamma)^{-1}(B)\in {\mathcal B}^\gamma$, where the latter space is identified 
in \cite{lemopa7} as the finite strain space of $S_\gamma$. Given
$V_\gamma$ and $B_\gamma$ as above, we proceed as in \cite[Theorem 2.2]{lemopa7}, 
as follows.
With a slight abuse of notation, we write:
\begin{equation}\label{marta}
\mathcal{Q}_2(x,F_{tan}) = \min\left\{\mathcal{Q}_3(F_{tan} + c\otimes e_3 + e_3\otimes c); ~~
c\in\mathbb{R}^3\right\}.
\end{equation}
The unique vector $c$, which attains the above minimum
will be denoted $c(x,F_{tan})$. By uniqueness, the map $c$ is linear in its 
second argument.  Also, for all $F\in {\mathbb R}^{3\times3}$, by $l(F)$ we denote 
the unique vector in ${\mathbb R}^3$, linearly depending on $F$,  for which:
$$\mbox{sym}\big(F - (F_{2\times 2})^*\big) 
= \mbox{sym}\big(l(F) \otimes e_3\big).$$

Recall  that $\gamma=h^\alpha$.
Given $B_{\gamma}\in\mathcal{B}^\gamma$, there exists a sequence of vector 
fields $w^h\in W^{1,2}(S_\gamma,\mathbb{R}^3)$
such that $\|\mbox{sym }\nabla w^h - B_\gamma\|_{L^2(S_\gamma)} $ converges to $0$.  
Without loss of generality, we may assume that $w^h$ are smooth, and
(by possibly reparameterizing the sequence) that:
\begin{equation}\label{norm}
\lim_{h\to 0} \sqrt{h}\|w^h\|_{W^{2,\infty}(S_\gamma)} = 0.
\end{equation}
In the same manner, we approximate $V_\gamma$ by a sequence $v^h\in W^{2,\infty}(S_\gamma,\mathbb{R}^3)$ such that, 
for a sufficiently small, fixed $\epsilon_0>0$:
\begin{equation}\label{vhapprox}
\begin{split}
& \lim_{h\to 0} \|v^h - V_\gamma\|_{W^{2,2}(S_\gamma)} = 0, 
\qquad h\|v^h\|_{W^{2,\infty}(S_\gamma)} \leq \epsilon_0,\\
&  \lim_{h\to 0}\frac{1}{h^2}~ \mu\left\{x\in S; ~~ v^h(x) \neq V(x)\right\} =0.
\end{split}
\end{equation}
The existence of such $v^h$ follows by partition of unity and a truncation
argument, as a special case
of the Lusin-type result for Sobolev functions
(see \cite[Proposition 2]{FJMhier}).

Finally, define the sequence of  deformations $u^h\in W^{1,2}((S_\gamma)^{h},\mathbb{R}^3)$ by:
\begin{equation}\label{rec_seq}
\begin{split}
u^h(x+t\vec n^\gamma) & = x + h v^h(x) + h^2 w^h(x) \\
& \qquad + {t}\vec n^\gamma(x) + {t}h\Big(\Pi_\gamma v^h_{tan} - \nabla (v^h\vec n^\gamma)\Big)(x)\\
& \qquad - t h^2 (\vec n^\gamma)^T \nabla w^h
+ {t} h^2 d^{0,h}(x) + \frac 12 t^2h d^{1,h}(x).
\end{split}
\end{equation}
The vector fields $d^{0,h}, d^{1,h}\in W^{1,\infty}(S_\gamma,\mathbb{R}^3)$ are defined so that:
\begin{equation}\label{nd01h}
\lim_{h\to 0} \sqrt{h} \left(\|d^{0,h}\|_{W^{1,\infty}(S_\gamma)} + \|d^{1,h}\|_{W^{1,\infty}(S_\gamma)}\right) = 0
\end{equation}
and that, in $L^2(\Omega)$:
\begin{equation}\label{warp}
\begin{split}
\lim_{h\to 0} d^{0,h} \circ \phi_\gamma & =  l(\epsilon_g) - \frac 12 |\nabla v|^2 e_3  
+  c \Big(B -  \frac 12 \nabla v \otimes \nabla v 
- (\mbox{sym } \epsilon_g)_{2\times2} \Big),\\
\lim_{h\to 0} d^{1,h} \circ \phi_\gamma& =  l(\kappa_g) + c\Big(-\nabla^2 v - (\mbox{sym }\kappa_g)_{2\times 2}\Big).
\end{split}
\end{equation}
Now, the convergence statements of Theorem  \ref{shallow<1-limsup}
are verified by straightforward calculations as in the proofs of
\cite[Theorem 2.2]{lemopa7} and \cite[Theorem 1.4] {lemapa}.  

\begin{remark}
One may define the recovery sequence explicitly on $\Omega$, without
the diagonal argument presented in the proof above. Namely we proceed as follows. 
 
We approximate $V=(\tilde w, v)$ by a sequence $V^h = (\tilde w^h, v^h) \in W^{2,\infty}(\Omega,\mathbb{R}^3)$  
such that, for a sufficiently small, fixed $\epsilon_0>0$:
\begin{equation}\label{vhapprox}
\begin{split}
& \lim_{h\to 0} \|V^h - V\|_{W^{2,2}(\Omega)} = 0, 
\qquad h\|V^h\|_{W^{2,\infty}(\Omega)} \leq \epsilon_0,\\
&  \lim_{h\to 0}\frac{1}{h^2}~ \mu\left\{x\in \Omega; ~~ V^h(x) \neq V(x)\right\} =0.
\end{split}
\end{equation}
Also, let $w^h:\Omega\to {\mathbb R}^3$ be such that 
$$ \displaystyle B = \lim_{h\to 0}   
\Big [\mbox{sym} \nabla(w^h_1, w^h_2) + \mbox{sym} (\nabla w^h_3
\otimes \nabla v_0) \Big ]. $$  
Without loss of generality, we may assume that $w^h$ are smooth, and that:
\begin{equation}\label{norm}
\lim_{h\to 0} \sqrt{h}\|w^h\|_{W^{2,\infty}(\Omega)} = 0.
\end{equation} The recovery sequence is then given by: 
\begin{equation}\label{recoveryseq}
\begin{split} 
y^h(x, t) & = u^h (x+ h^\alpha v_0(x) e_3 + th\vec n^\gamma) \\ & 
= \left[\begin{array}{c}x \\ h^\alpha v_0 (x) \end{array}\right] 
+ h \left[\begin{array}{c} h^{\alpha} \tilde w^h(x) \\ v^h(x)\end{array}\right] +  
h^2 \left[\begin{array}{c}    w^h_{tan}  \\ h^{-\alpha}
    w^h_3 \end{array}\right] \\ &  \qquad \qquad 
+ th \left[\begin{array}{c}-h^\alpha\nabla v_0(x)\\1\end{array}\right]  
+ th \left[\begin{array}{c}-h\nabla v^h(x)\\1\end{array}\right]   \\ &
\qquad \qquad  -th^3 
\left[\begin{array}{c} 0 \\  h^{-\alpha} \nabla w^h_3 \end{array}\right]
+  h^3 t d^{0,h}(x) + \frac{1}{2}h^3 t^2 d^{1,h}(x),
\end{split} 
\end{equation} where the vector fields $d^{0,h}, d^{1,h}\in W^{1,\infty}(\Omega,\mathbb{R}^3)$ are defined similarly as before.

\end{remark}

\end{document}